\documentclass[12pt,draftcls,onecolumn,romanappendices]{IEEEtran}

\usepackage[utf8x]{inputenc}

\usepackage{graphicx, color}      
\usepackage{amsmath}
\usepackage{amssymb}
\usepackage[dvips]{epsfig}
\usepackage{amsthm}
\usepackage{enumerate}
\usepackage{color}
\usepackage{epstopdf}

\newcommand{\N}        {\mathbb{N}}
\newcommand{\R}        {\mathbb{R}}
\newcommand{\C}        {\mathbb{C}}
\newcommand{\CC}        {\mathfrak{C}}
\newcommand{\w} {{\tt w}}
\newcommand{\m} {{\tt m}}
\newcommand{\p} {{\tt p}}
\newcommand{\n} {{\tt n}}

\newcommand{\Lw}      {\mathfrak{L}^{\tt{w}}}
\newcommand{\Cinf}        {\mathfrak{C}^\infty}

\newcommand{\B}{\mathfrak{B}}

\newcommand{\ddt} {\frac{d}{dt}}

\newcommand{\der} {\frac{d}{dt}}
\newcommand{\derb} {\left(\frac{d}{dt}\right)}

\newcommand{\col}{\mbox{\rm col}}
\newcommand{\rank}{\mbox{\rm rank}}

\newtheorem{lemma}{Lemma}
\newtheorem{theorem}{Theorem}
\newtheorem{definition}{Definition}
\newtheorem{example}{Example}
\newtheorem{corollary}{Corollary}
\newtheorem{proposition}{Proposition}

\newtheorem{remark}{Remark}
\usepackage{listings}

\begin{document}
\title{Stability of switched  linear differential systems}
\author{J.C.~Mayo-Maldonado, P.~Rapisarda and~P.~Rocha\thanks{J.C. Mayo-Maldonado and P. Rapisarda are with the CSPC group, School of Electronics and Computer Science, University of Southampton, Great Britain,  e-mail: {\tt jcmm1g11,pr3@ecs.soton.ac.uk}, Tel:  +(44)2380593367, Fax: +(44)2380594498.}\thanks{P. Rocha is with the Department of Electrical and Computer Engineering, Faculty of Engineering, University of Oporto, Portugal, e-mail: {\tt mprocha@fe.up.pt}, Tel: +(351)225081844, Fax: +(351)225081443.}}

\maketitle

\begin{abstract}
We study the stability of switched systems where the dynamic modes are described by systems of higher-order linear differential equations not necessarily sharing the same state space. Concatenability of trajectories at the switching instants is specified by {gluing conditions}, i.e. algebraic conditions on the trajectories and their derivatives at the switching instant. We provide sufficient conditions for stability based on LMIs  for systems with general gluing conditions. We also analyse the role of positive-realness in providing sufficient polynomial-algebraic conditions for stability of two-modes switched systems with special gluing conditions. 
\end{abstract}


\section{Introduction}

In established approaches, switched systems consist of a  bank of state-space or descriptor form representations (see \cite{HeMo,liberzon,trennThesis,trenn}) sharing a common global state space, together with a supervisory system  determining  which of the modes is active. 
In many situations, modelling switched systems with state representations sharing a common state  is justified from first principles. For example, when dealing with  switched electrical circuits, it can be necessary to consider the state of the overall circuit in order to model the transitions between the different dynamical regimes. However, in other situations modelling a switched system using a common global state space is not justified by physical considerations. For example, in a \emph{multi-controller} control system consisting of a plant and a bank of controllers which have different orders, the dynamical regimes have different state space dimensions. Such a system can be modelled using a global state space common to the different dynamics; however, there is no compelling reason to do so, since at any given time only one controller is active.  
In \emph{hybrid renewable energy conversion systems} (see e.g. \cite{renew}) several energy sources are connected to power devices in order to transform and deliver energy to a grid. Due to the intermittent nature of renewable energies, the need arises to  connect or disconnect dynamical conversion systems such as wind turbines, photovoltaic/fuel cells, etc., whose mathematical models  have different orders.  A similar situation arises in \emph{distributed power systems} \cite{fredlee}, where different electrical loads are connected or disconnected from a power source. 

Modelling such systems using a global state variable results in a more complex (more variables and more equations) dynamical model than alternative representations. For instance, such a description of a distributed power system would include the state variables of each possible load, even though in general not all loads are connected at the same time and contributing to the dynamics of the overall system. This approach also scores low on modularity, i.e.  the independent development and incremental combination of models.

Another issue with the classical approach to switched systems is that modelling from first principles usually does not yield a
state-space description (for a detailed elaboration of this position  see \cite{BehMod}). A system is the  interconnection of subsystems; to model it one first describes the subsystems and the interconnection laws, possibly hierarchically repeating such procedure until simple representations (e.g. derived from a library or from elementary physical principles) can be used. Such a model typically involves  algebraic relations (e.g. kinematic or equilibrium constraints), and differential equations of first- and  second-order (e.g., constitutive equations of electrical components,  dynamics of masses), or of higher-order (e.g., resulting from the elimination of auxiliary variables). 

These considerations motivate the development of a framework to model and analyse switched systems using higher-order models describing  dynamics with  different complexity. In our approach, each dynamic mode is associated with a \emph{mode behaviour}, the set of trajectories that satisfy the dynamical laws of that mode. A \emph{switching signal} determines when a transition between dynamic modes occurs. To be   \emph{admissible}  for the switched behaviour, a trajectory must satisfy two conditions. Firstly,  it must satisfy the laws of the mode active in  the interval between two consecutive switching instants. Secondly, at the switching instants the trajectory must satisfy certain \emph{gluing conditions},  representing the physical constraints imposed by the switch, e.g. conservation of charge, kinematic constraints, and so forth. The set of all admissible trajectories is the \emph{switched behaviour}, and is the central object of study in our framework.  

Following the preliminary investigations for systems with one variable reported in \cite{mtns12,cdcpaper}, in this paper we propose such a framework for the linear multivariable autonomous case. Each mode behaviour is represented by a set of linear, constant-coefficient higher-order differential equations. The gluing conditions consist of algebraic equations involving the values of a trajectory and its derivatives before and after the switching instant. 
We focus on {closed systems}, i.e. systems without input variables, and we study their Lyapunov stability using quadratic functionals of the system variables and their derivatives. We present new sufficient conditions based on systems of LMIs for the existence of a higher-order quadratic Lyapunov function for arbitrary gluing conditions. Such systems of LMIs can be set up straightforwardly from the equations of the modes and the gluing conditions. We also study the relation of positive-realness with the stability of a class of (``standard'') two-modes switched systems; these conditions are multivariable generalisations of those presented in the scalar case in \cite{mtns12,cdcpaper}. Finally, we introduce the notion of positive-real completion of a given transfer function. 

Following the  behavioural approach for linear systems (see \cite{yellow}), the mode equations and the gluing conditions are represented by one-variable polynomial matrices,  and the Lyapunov functionals by two-variable ones. The calculus of such functionals and representations is a powerful tool conducive to the use of computer algebra techniques for the modelling and analysis of switched systems.

The  approaches closest to ours are those of Geerts and Schumacher (see \cite{GS1,GS2}) on impulsive-smooth systems and polynomial representations; and that of Trenn about  linear differential-algebraic equations (DAE's; see \cite{liberzontrenn09,trennThesis,trenn,trenn13}), and most pertinently his recent publication \cite{TW12} (also worth mentioning is \cite{BM12}, which however is less related to our setting). These authors consider mode dynamics with different state-space dimension, a situation  generally involving impulses in the system trajectories, a relevant issue also for practical reasons (see e.g. \cite{dgarciatrenn10}). In \cite{trennThesis}  a unifying, rigorous distributional framework for switched systems has been given. When the  modes are described by DAE's this approach encompasses also the detection of impulses {directly} from the equations; for higher-order representations as in \cite{TW12}, \emph{impact maps} are used instead to specify {explicitly} the impulsive part of the behaviour. Stability (also in a Lyapunov sense) for impulse-free switched DAE's has been investigated in \cite{liberzontrenn09}.   

In this paper we deal with autonomous (i.e. closed) modes; impulsive effects are  \emph{implicitly} defined by the gluing conditions and the mode dynamics involved in the transition (i.e. do not depend for example on the degree of differentiability of some input variable). Our position is that gluing conditions are a given; we take them at face value. Whether they are physically meaningful or not; whether they imply impulses or not; and whether the latter is an important issue for the particular physical system at hand, are major modelling issues that we assume have been weighed carefully by the modeller (on this issue see also p. 749 of \cite{GS1}).  In certain cases, see Examples \ref{ex:elcirc} and \ref{ex:math} below, our attitude towards gluing conditions seems to  involve less conceptual difficulties than letting the equations to dictate the re-initialisation mechanism at the switching instants. This ``agnostic" position does not absolve us though from the important task, relevant for instance in the case of models assembled from libraries,  of studying how to determine the presence of impulses directly from the equations and associated gluing conditions; this is a pressing research question to be considered elsewhere (on this issue see \cite{trennThesis,trenn}).  

We study stability for higher-order representations also in the presence of impulses. We recognise the validity of the position taken in \cite{trenn} for switched DAE's that when impulses are allowed, the idea that in a stable system small initial states produce  state trajectories vanishing at infinity is awkward. However, we also notice that the impulsive nature of  solutions has not discouraged investigation of stability (also with Lyapunov methods), both in the classical impulsive differential equations framework (see e.g. Ch. 3 of \cite{Samolienko}), and in the state-space approach to switched systems, where impulses are implicit in the reset maps (see e.g. \cite{HeMo}). Other recent approaches are focused on the study of switched systems whose trajectories are everywhere continuous, and thus not contain impulses; e.g. \cite{DissZhao}, where a complete framework for dissipative switched systems is presented (see Sec. II \emph{ibid}.).

The paper is organised as follows: in section \ref{sec:SLDS} we define switched linear differential systems (SLDS),  we give  examples of SLDS, and we discuss the issue of well-posedness. In section \ref{sec:stab} we give sufficient conditions for stability of a SLDS based on the existence of a multiple Lyapunov function (MLF). We also discuss how to compute MLFs using LMIs. In section \ref{sec:standard} we focus on two-modes SLDS, and we investigate the role of positive-realness in establishing the stability of such systems. The notational conventions and some background material on the behavioural approach and quadratic differential forms is gathered in Appendix I, while the proofs are gathered in Appendix II. 

\section{Switched Autonomous Linear Differential Systems}\label{sec:SLDS}

\subsection{Basic definitions}
Recall from App. \ref{sec:behbasics} the definition of $\mathfrak{L}^{\tt w}$ as the set of linear differential behaviours. A switched linear differential system is defined in the following way (see also \cite{mtns12,cdcpaper}).

\begin{definition}\label{def:SLDS}\rm
A \emph{switched linear differential system (SLDS)} $\Sigma$ is a quadruple $\Sigma=\{\mathcal{P},\mathcal{F},\mathcal{S},\mathcal{G}\}$ where
${\mathcal P}= \{1, \dots, N \}\subset \mathbb{N}$, is the set of \emph{indices}; ${\mathcal F} = \left\{ \mathfrak{B}_1, \ldots, \mathfrak{B}_N \right\}$, with $\mathfrak{B}_j\in \mathfrak{L}^{\tt w}$ for $j \in \mathcal P $, is the \emph{bank of behaviours}; ${\mathcal S}=\{ s:\R \rightarrow \mathcal{P} \mid s\mbox{ is piecewise constant and right-continuous} \}$, is the set of admissible \emph{switching signals}; and 
\[
\mathcal{G}=\left\{(G_{k\rightarrow\ell}^-(\xi), G_{k\rightarrow\ell}^+(\xi))\in \R^{\bullet \times \tt w}[\xi]\times \R^{\bullet \times \tt w}[\xi] \mid 1\leq k,\ell\leq N\; , \; k\neq \ell\right\}
\] 
is the set of \emph{gluing conditions}. The set of \emph{switching instants} associated with $s\in \mathcal{S}$ is defined by $\mathbb{T}_s:=\{ t \in \R ~|~ \lim_{\tau \nearrow t} s(\tau) \neq s(t) \} = \{ t_1,t_2,\dots \}$, where $t_i<t_{i+1}$.
\end{definition}

A SLDS induces a switched behaviour, defined as follows. 
\begin{definition}\label{def:SLDB}\rm
Let $\Sigma=\{\mathcal{P},\mathcal{F},\mathcal{S},\mathcal{G}\}$ be a SLDS, and let $s \in \mathcal{S}$. The \emph{$s$-switched linear differential behaviour $\mathfrak{B}^s$} is the set of trajectories $w:\R\rightarrow \R^\w$ that satisfy the following two conditions:
\begin{enumerate}
  \item for all  $ t_i, t_{i+1} \in \mathbb{T}_s$, there exists $\mathfrak{B}_k\in\mathcal{F}$, $ k \in \mathcal{P}$ such that $w\!\mid_{ [t_i, t_{i+1}) } \in \mathfrak{B}_k\!\mid_{[t_i, t_{i+1})}$;   \item $w$ satisfies the gluing conditions $\mathcal{G}$ at the switching instants for each $t_i\in \mathbb{T}_s$, i.e. \\ $(G^+_{s(t_{i-1})\rightarrow s(t_i)}(\frac{d}{dt}))w(t_i^+ ) = (G_{s(t_{i-1})\rightarrow s(t_i)}^-(\frac{d}{dt}))w(t_i^-)$. \end{enumerate}
The \emph{switched linear differential behaviour (SLDB)}  $\mathfrak{B}^\Sigma$ of $\Sigma$ is defined by $\mathfrak{B}^\Sigma:= \bigcup_{s\in \mathcal{S}} \mathfrak{B}^s$.
\end{definition}
We make the standard assumption (see e.g. sect. 1.3.3 of \cite{sunge}) that   the number of switching instants in any finite interval of $\R$ is \emph{finite}. Moreover, in this paper we assume that the behaviours $\B_i$, $i\in\mathcal{P}$ are \emph{autonomous}. Since the trajectories of an autonomous behaviour are infinitely differentiable (see 3.2.16 of \cite{yellow}), the trajectories of a switched behaviour as in Def. \ref{def:SLDB} are smooth in any interval between two consecutive switching times.

We now give three examples of switched behaviours; besides exemplifying the Definitions, they allow us to point out some important features of our approach to switched systems (see also section \ref{sec:stab} for another more realistic example). 
\begin{example}\label{ex:elcirc}\rm 
Consider the  electrical circuit  in Fig. \ref{Fig:electricalcircuit}, where $C=1~F$, $R=1~\Omega$, and  $w_1$ and $w_2$ are  voltages. 
\begin{figure}[htbp]
\includegraphics[scale=.7]{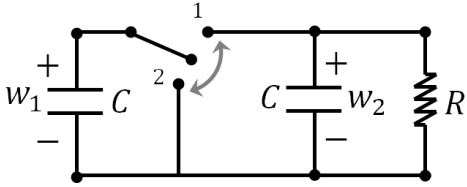}
\caption{An electrical circuit}
\label{Fig:electricalcircuit}
\end{figure}
With the switch in position 1, the dynamical equations are 
\begin{eqnarray}\label{eq:B1}
\frac{d}{dt}w_2+w_2&=&0\nonumber\\
w_1-w_2&=&0\; ;
\end{eqnarray}
 when the switch is in position 2, the dynamical equations are 
\begin{eqnarray}\label{eq:B2ST}
\frac{d}{dt}w_2+w_2&=&0\nonumber\\
w_1&=&0\; . 
\end{eqnarray}
The gluing conditions follow from the principle of conservation of charge (see also \cite{CTV}): for a transition $\B_2 \rightarrow \B_1$ the matrices are
\begin{equation}\label{eq:gc2to1}
G_{2\rightarrow 1}^-:=\begin{bmatrix} 
0&\frac{1}{2}\\0&\frac{1}{2}
\end{bmatrix}\; ,\; G_{2\rightarrow 1}^+:=I_2\; , 
\end{equation}
and for  a transition $\B_1\rightarrow \B_2$ they are 
\begin{equation}\label{eq:gc1to2}
G_{1\rightarrow 2}^-:=\begin{bmatrix} 
0&0\\0&1
\end{bmatrix}\; ,\; G_{1\rightarrow 2}^+:=I_2\; . 
\end{equation}
The switched behaviour consists of all piecewise smooth functions $\mbox{\rm col}(w_1,w_2)$ that satisfy (\ref{eq:B1}) or (\ref{eq:B2ST}) depending on the position of the switch, and that at the switching instant satisfy either $w_1(0^+)=\frac{1}{2}w_2(0^-)$, $w_2(0^+)=\frac{1}{2} w_2(0^-)$ (for a transition $\B_2 \rightarrow \B_1$) or $w_1(0^+)=0$, $w_2(0^+)=w_2(0^-)$ (for a transition $\B_1 \rightarrow \B_2$). These gluing conditions imply that in any non-trivial case the value of $w_1$ jumps at the switching instant. \qed
\end{example}

\begin{example}\label{ex:conccond}\rm 
Depending on the value of a switching signal a plant $\Sigma_P$ with two external variables, described by the differential equation $\der w_1-w_1-w_2=0$, is   connected with one of two possible controllers $\Sigma_{C_1}$ and $\Sigma_{C_2}$, described respectively by $-3\der w_1-w_1-\der w_2=0$ and $-2 w_1-w_2=0$. Depending on which controller is active, the resulting closed-loop behaviours are  
\[
\B_1:=\ker \begin{bmatrix} \der-1&-1\\
-3\der-1&-\der\end{bmatrix}\; \mbox{ and } \B_2:=\ker \begin{bmatrix} \der-1&-1\\
-2&-1 \end{bmatrix}\; .
\] 
Note that $\B_1$ and $\B_2$ have different McMillan degree (2 and 1, respectively).  We define the gluing conditions for the SLDS associated with $\B_1$ and $\B_2$ by 
\[
G_{2\rightarrow 1}^-(\xi):=\begin{bmatrix} 0&1\\ 0&-2
 \end{bmatrix}\; ,\; G_{2\rightarrow 1}^+(\xi):=\begin{bmatrix} 0&1\\ 1&0 \end{bmatrix}
 \] 
 and by 
 \[
 G_{1\rightarrow 2}^-(\xi):=\begin{bmatrix}1&0 \end{bmatrix}\; , G_{1\rightarrow 2}^+(\xi):=\begin{bmatrix}1&0 \end{bmatrix}\; . 
 \]
 The rationale underlying our choice of gluing conditions is that any trajectory of $\B_1$  is uniquely specified by the instantaneous values of $\mbox{\rm col}(w_1,w_2)$, while a trajectory of $\B_2$ is uniquely specified by the instantaneous value of $w_1$. Moreover, when switching from the dynamics of $\B_1$ to those of $\B_2$, we require that the values of $w_1$ before and after the switching instant coincide. In a switch from $\B_2$ to $\B_1$,  since the second differential equation describing $\B_2$ yields $w_2=-2 w_1$ before the switch, we impose that $w_2(t_k^+)=w_2(t_k^-)=-2w_1(t_k^-)$.  This makes the switched trajectory as smooth as possible, taking into account the restrictions imposed by each individual behaviour $\B_1$ and $\B_2$. \qed
\end{example}

\begin{example}\label{ex:math}\rm
Consider two behaviours respectively described by the equations 
\begin{eqnarray}\label{eq:B1p}
\frac{d}{dt}w_2+w_2&=&0\nonumber\\
w_1-w_2&=&0\; ,
\end{eqnarray}
and  
\begin{eqnarray}\label{eq:B2p}
\frac{d}{dt}w_1+\frac{d}{dt}w_2+w_1+w_2&=&0\nonumber\\
w_1&=&0\; . 
\end{eqnarray}
The gluing conditions for a transition $\B_2 \rightarrow \B_1$ are associated with the matrices 
\begin{equation}\label{eq:gc2to1p}
G_{2\rightarrow 1}^-:=\begin{bmatrix} 
0&1\\0&1
\end{bmatrix}\; ,\; G_{2\rightarrow 1}^+:=I_2\; , 
\end{equation}
and for  a transition $\B_1\rightarrow \B_2$ they are defined by 
\begin{equation}\label{eq:gc1to2p}
G_{1\rightarrow 2}^-:=\begin{bmatrix} 
0&0\\\frac{1}{2}&\frac{1}{2}
\end{bmatrix}\; ,\; G_{1\rightarrow 2}^+:=I_2\; ; 
\end{equation}
i.e. in a switch  $\B_1\rightarrow\B_2$ the new value of $w_2$ is the average of the old values of $w_1$ and $w_2$. \qed
\end{example}

Examples \ref{ex:elcirc} and \ref{ex:math} offer the opportunity of making two  important remarks. 

\begin{remark}\rm 
Note that (\ref{eq:B2ST}) and  (\ref{eq:B2p}) describe the same set of solutions; indeed, the description (\ref{eq:B2ST}) can be obtained from (\ref{eq:B2p}) by \emph{unimodular} operations, which in the case of autonomous systems do not alter the solution set (see Th. 2.5.4 and Th. 3.2.16 of \cite{yellow})\footnote{Some equivalence results for the $\Cinf$-case are not valid for non-autonomous systems and $\mathcal{L}_1^{\rm \tiny loc}$ trajectories; see for example \cite{propelim}. On equivalence of polynomial representations of switched systems, see sect. 3 of \cite{GS2}.}. 
Considering that (\ref{eq:B1}) and  (\ref{eq:B1p}) are the same equation, the dynamic modes are the same for both switched systems; thus the  two switched behaviours are \emph{different because the gluing conditions are}. We will prove later in this paper that these two switched systems also have different stability properties- that of Ex. \ref{ex:elcirc} is stable under arbitrary switching signals, while the other is not. Stability arises from the \emph{interplay} of mode dynamics and gluing conditions. \qed
\end{remark}

\begin{remark}\rm
Gluing conditions should be defined on the basis of  the physics of the system under study. Those for the system of Example \ref{ex:elcirc} are meaningful for the particular physical system at hand. However, for \emph{another}  physical system whose modes happen to be  described also by (\ref{eq:B1p})-(\ref{eq:B2p}), the  conditions (\ref{eq:gc2to1p})-(\ref{eq:gc1to2p}) may also be physically plausible. In each case we assume that well-grounded physical considerations have been motivating the choice. 
\qed
\end{remark}

\subsection{Well-posedness of gluing conditions}\label{sec:wellpos}

In principle Def.s \ref{def:SLDS} and \ref{def:SLDB} do not restrict the gluing conditions; however, since we assume that the modes  are autonomous, i.e. no external influences are applied to the system  between consecutive switching times, it is reasonable to require more. Namely, no  different admissible trajectories should exist with the same past (i.e. same mode transitions at the same switching instants, and same restrictions from $t=-\infty$ up until a given  switching instant $\overline{t}$). If such trajectories  exist, then at $\overline{t}$ the past  ``splits" in different futures; however, since no external inputs  could trigger such a change,  the past of a trajectory should uniquely define its future.  These considerations lead to the concept of \emph{well-posed} gluing conditions, which we now introduce. 

In order to do so, we first fix kernel representations $\B_k=:\ker R_k\derb$, with $R_k\in\R^{\tt w \times w}[\xi]$ nonsingular, $k=1,...,N$ for the modes. 
We also define $n_k:=\mbox{deg}(\mbox{det}(R_k))$, $k=1,...,N$,
and we fix  minimal state maps (see App. \ref{app:backmater:statemaps}) $X_k\in\R^{n_k\times \w}[\xi]$, $k=1,...,N$. Every polynomial differential operator $G\derb$ on $\B_k$ has a unique $R_k$-\emph{canonical representative} $G^\prime\derb$, denoted by $G^\prime=G \mod R_k$, such that  $G^\prime\derb w=G\derb w$ for all $w\in\B_k$ (see App. \ref{sec:behbasics}). Now let  $\left(G_{k\rightarrow\ell}^-,G_{k\rightarrow\ell}^+ \right)\in\mathcal{G}$; then   $\left(G_{k\rightarrow\ell}^- \mod R_k,G_{k\rightarrow\ell}^+ \mod R_\ell\right)$ are equivalent to $\left(G_{k\rightarrow\ell}^-,G_{k\rightarrow\ell}^+ \right)$, in the sense that the algebraic conditions imposed by the one pair are satisfied iff they are satisfied by the other. Moreover, since $G_{k\rightarrow\ell}^- \mod R_k$ and $G_{k\rightarrow\ell}^+ \mod R_\ell$ are $R_k$-, respectively $R_\ell$-canonical, there exist constant matrices $F_{k\rightarrow\ell}^-$ and $F_{k\rightarrow\ell}^+$ of suitable dimensions such that  $G_{k\rightarrow\ell}^-(\xi)\;\mbox{mod}\;R_k=F_{k\rightarrow\ell}^-X_k(\xi)$ and $G_{k\rightarrow\ell}^+(\xi)\;\mbox{mod}\;R_\ell=F_{k\rightarrow\ell}^+X_\ell(\xi)$. We  call 
\[
\mathcal{G}^\prime:=\{(F_{k\rightarrow\ell}^-X_k(\xi),F_{k\rightarrow\ell}^+X_\ell(\xi))\mid 1\leq k,\ell\leq N, k\neq \ell \}
\]
 the \emph{normal form} of $\mathcal{G}$.   
\begin{definition}\label{def:WPGC}\rm
Let $\Sigma$ be a SLDS with $\B_i=\ker R_i\derb$ autonomous, $i=1,\ldots,N$. The normal form gluing conditions $\mathcal{G}^\prime:=\{(F_{k\rightarrow\ell}^-X_k(\xi),F_{k\rightarrow\ell}^+X_\ell(\xi))\}_{k,\ell=1,\ldots,N, k\neq \ell}$ are \emph{well-posed} if for all $k,\ell=1,\ldots,N$, $k\neq \ell$, and for all $v_k\in\R^{n_k}$ there exists at most one $v_\ell\in\R^{n_\ell}$ such that $F_{k\rightarrow\ell}^-v_k=F_{k\rightarrow\ell}^+v_\ell$. 
\end{definition}
Thus if a transition occurs between $\B_k$ and $\B_\ell$  at  $t_j$, and if an admissible trajectory ends at a ``final state" $v_k=X_k\left(\frac{d}{dt} \right)w(t_j^-)$, then there exists at most one ``initial state" for $\B_\ell$, defined by $X_\ell\left(\frac{d}{dt} \right)w(t_j^+):=v_\ell$, compatible with the gluing conditions. 

\begin{example}\rm
Consider the gluing conditions of Example \ref{ex:conccond}.  A minimal state map for $\B_1$ is $X_1(\xi):=I_2$, and a minimal state map for $\B_2$ is $X_2(\xi)=\begin{bmatrix}1&0 \end{bmatrix}$.  It follows that $G_{2\rightarrow 1}^+\mod R_1(\xi)=G_{2\rightarrow 1}^+(\xi)=F_{2\rightarrow 1}^+ X_1(\xi):=\begin{bmatrix} 0&1\\ 1&0\end{bmatrix} I_2$. Moreover, $G_{1\rightarrow 2}^+\mod R_2(\xi)=G_{1\rightarrow 2}^+(\xi)=F_{1\rightarrow 2}^+ X_2(\xi):=1 \begin{bmatrix} 1&0\end{bmatrix}$. These gluing conditions are well-posed. It can be verified in a similar way that the gluing conditions of Examples \ref{ex:elcirc} and \ref{ex:math} are also well-posed. \qed
\end{example}

%

%

\begin{remark}\label{rem:existence}\rm 
Well-posedness only concerns \emph{uniqueness}, not  \emph{existence} of an admissible ``initial condition" $v_\ell$ in $\B_\ell$ for a given ``final condition" $v_k$ in $\B_k$. 
It may happen that the gluing conditions cannot be satisfied by nonzero trajectories; they may not be ``consistent" with the mode dynamics. For example, consider a SLDS with modes  (\ref{eq:B1p}) and (\ref{eq:B2p}),  and  (well-posed) gluing conditions $G_{2\rightarrow 1}^-:=I_2,G_{2\rightarrow 1}^+:=I_2$, $G_{1\rightarrow 2}^-:=I_2,G_{1\rightarrow 2}^+:=I_2$. $w\in\B_1$ iff $w(t)=\alpha~ \mbox{\rm col}(e^{-t},e^{-t})$, $\alpha\in\R$; and  $w\in\B_2$ iff $w(t)=\alpha~ \mbox{\rm col}(e^{-t},0)$, $\alpha\in\R$. Since constant switching signals $\sigma_1=1$ and $\sigma_2=2$ are admissible, it follows that $\B^\Sigma\supset \B_i$, $i=1,2$. However, no genuine switched trajectory exists besides the zero one, since the gluing conditions cannot be satisfied by nonzero trajectories of either of the behaviours.

The problem whether a given ``initial condition"  is consistent or not with the mode dynamics was solved most satisfactorily in the switched DAE's framework of Trenn (see Ch. 4 of \cite{trennThesis}); algorithms are stated that from the matrices describing a mode compute ``consistency projectors" whose image is  the subspace of consistent initial values. We briefly discuss the issue of consistent gluing conditions in our framework. 

Denote the roots of $\det~R_k(\xi)$  by $\lambda_{k,i}$, $i=1,\ldots n_k$. We  assume for ease of exposition that the algebraic multiplicity of $\lambda_{k,i}$ equals  the dimension of $\ker R_k(\lambda_{k,i})$. It follows from sect. 3.2.2 of \cite{yellow} that $w\in\B_k$ iff there exist $\alpha_{k,i}\in\C$, $i=1,\ldots,n_i$  such that $w=\sum_{i=1}^{n_k} \alpha_{k,i} w_{k,i} \exp_{\lambda_i t}$, where $w_{k,i}\in\C^{\w}$ is such that $R_k(\lambda_{k,i})w_{k,i}=0$, and the $w_{k,i}$ associated with equal $\lambda_{k,i}$ are linearly independent. Note that the $\alpha_{k,i}$ associated to conjugate $\lambda_{k,i}$ are conjugate.


Now define $V_i:=\begin{bmatrix}X_i(\lambda_{i,1})w_{i,1}& \ldots& X_i(\lambda_{i,n_i})w_{i,n_i}\end{bmatrix}\in\C^{n_i\times n_i}$ and $\alpha_i:=\begin{bmatrix} \alpha_{i,1}&\ldots& \alpha_{i,n_i}\end{bmatrix}^\top$,  $i=k,\ell$;  and consider a switch from $\B_k$ to $\B_\ell$ at $t=0$. The gluing conditions stipulate that $G_{k\rightarrow\ell}^-(w)(0^-)=F_{k\rightarrow\ell}^-V_k\alpha_k=F_{k\rightarrow\ell}^+ V_\ell\alpha_\ell=G_{k\rightarrow\ell}^+(w)(0^+)$. Nonzero  $\alpha_{i}$, $i=k,\ell$  exist satisfying this equality if and only if  $\mbox{im}~F_{k\rightarrow\ell}^-V_k \subseteq
\mbox{im}~F_{k\rightarrow\ell}^+V_\ell$. Standard arguments in ordinary differential equations show that $V_k$ and $V_\ell$ are nonsingular; consequently the consistency condition can be equivalently stated as $\mbox{im}~F_{k\rightarrow\ell}^- \subseteq 
\mbox{im}~F_{k\rightarrow\ell}^+$. \qed
\end{remark}

Well-posedness implies that for all $k,\ell=1,\ldots,N$, $k\neq \ell$, $F_{k\rightarrow\ell}^+$ is full column rank, and consequently there exists a \emph{re-initialisation map} $L_{k\rightarrow \ell}:  \R^{n_k} \rightarrow \R^{n_\ell}$ defined by  $L_{k\rightarrow \ell}:={F}_{k\rightarrow\ell}^{+\ast} F_{k\rightarrow\ell}^-$, where ${F}_{k\rightarrow\ell}^{+\ast}$ is a left inverse of $F_{k\rightarrow\ell}^{+}$. For all $t_j\in\mathbb{T}_s$ and all admissible  $w\in\B^\Sigma$ it holds that
\begin{eqnarray*}
&&\left[s(t_{j-1})=k, s(t_j)=\ell \right]~ \mbox{\rm and }
\left[G_{k\rightarrow \ell}^+\left(\frac{d}{dt} \right)w(t_j^+)=G_{k\rightarrow \ell}^-\left(\frac{d}{dt} \right)w(t_j^-)\right] \\
&&\Longrightarrow \left[X_\ell\left(\frac{d}{dt} \right)w(t_j^+)=L_{k\rightarrow \ell}\left(X_k\left(\frac{d}{dt} \right)w(t_j^-)\right)\right]\; .
\end{eqnarray*}
Note that the re-initialisation map is not uniquely determined unless $F_{k\rightarrow\ell}^+$ is nonsingular. In the rest of the paper, we  assume well-posed gluing conditions with  fixed  re-initialisation maps. 

\section{Multiple Lyapunov functions for SLDS}\label{sec:stab}
We call a SLDB $\mathfrak{B}^\Sigma$ (and by extension, the SLDS $\Sigma$) \emph{asymptotically stable} if $\lim _{t\rightarrow \infty}w(t)=0$ for all $w \in \mathfrak{B}^\Sigma$. It follows from this definition that in an asymptotically stable SLDS, all mode behaviours $\B_i$ must be asymptotically stable and consequently autonomous (see \cite{yellow}, sec. 7.2). 

Asymptotic stability for linear differential behaviours can be proved by producing a \emph{higher-order quadratic Lyapunov function}, i.e. a quadratic differential function (QDF) $Q_\Psi$ such that $Q_{\Psi}\stackrel \B \geq 0$ and $\frac{d}{dt}Q_{\Psi}\stackrel \B < 0$, see sect. 4 of \cite{QDF}. The next result gives a sufficient condition for stability of SLDS in terms of quadratic \emph{multiple Lyapunov functions}  (\emph{MLFs}) (see e.g \cite{LT12} and sect. III.B of \cite{survey}). 
\begin{theorem}\label{th:MLFSLDS}
Let  $\Sigma$ be a SLDS (see Def. \ref{def:SLDS}). Assume that there exist QDFs $Q_{\Psi_i}$, $i=1,...,N$ such that 
\begin{itemize}
\item[1.] $Q_{\Psi_i}\stackrel{\B_i} \geq 0$, $i=1,...,N$;
\item[2.] $\frac{d}{dt}Q_{\Psi_i}\stackrel{\B_i}< 0$, $i=1,...,N$;
\item[3.] $\forall$ $w\in\B^\Sigma$ and $\forall$ $t_j\in\mathbb{T}_s$, 
$Q_{\Psi_{s(t_{j-1})}}(w)(t_j^-)\geq Q_{\Psi_{s(t_{j})}}(w)(t_j^+)$. 
\end{itemize}
Then $\Sigma$ is asymptotically stable. 
\end{theorem}
\begin{IEEEproof}See Appendix \ref{app:proofs}.
\end{IEEEproof}
Conditions 1 and 2 of Th. \ref{th:MLFSLDS} are equivalent to $Q_{\Psi_i}$ being a Lyapunov function for $\B_i$, $i=1,\ldots,N$. Condition 3 requires that  the value of the multiple functional associated to $Q_{\Psi_i}$, $i=1,...,N$, does not increase at the switching instants.

We now describe a procedure, based on the calculus of QDFs and on LMIs, to compute a MLF as in Th. \ref{th:MLFSLDS}. We first recall the following result from \cite{QDF}, that reduces the computation of quadratic Lyapunov functions to the solution of two-variable polynomial equations. 
\begin{theorem}\label{th:LyapfromQDF}
Let $\B=\ker R\derb$, with $R\in\R^{\w\times\w}[\xi]$ nonsingular. If $\B$ is asymptotically stable, for every $Q\in\R^{\bullet \times \w}[\xi]$ there exist $\Psi\in\R_s^{\w\times \w}[\zeta,\eta]$ and $Y\in\R^{\w \times \w}[\xi]$ such that $Q_\Psi\geq 0$ and 
\begin{equation}\label{eq:PLE}
(\zeta+\eta)\Psi(\zeta,\eta)=Y(\zeta)^\top R(\eta)+R(\zeta)^\top Y(\eta)-Q(\zeta)^\top Q(\eta)\; .
\end{equation}
If either one of $Q$ or $Y$ is $R$-canonical, then also the other and $\Psi$ are $R$-canonical. Moreover if $\mbox{\rm rank}~\mbox{\rm col}(R(\lambda), Q(\lambda))=\w$ for all $\lambda\in\C$ such that $\det R(\lambda)=0$, then $Q_\Psi\overset{\B}{>}0$. 
\end{theorem}
\begin{IEEEproof}The result follows from Th. 4.8 and Th. 4.12 of \cite{QDF}.
\end{IEEEproof}
Thus, a quadratic Lyapunov function $Q_\Psi$ can be computed by choosing some $Q$ and solving   the \emph{polynomial Lyapunov equation (PLE)} (\ref{eq:PLE}). Algebraic methods for solving it are illustrated in \cite{PLE}; we devise an LMI-based one more suitable to our purposes. We first relate (\ref{eq:PLE}) with a \emph{matrix} equation. 
\begin{proposition}\label{prop:PLEviaLMI}
Let  $\B=\ker R\derb$, with $R\in\R^{\tt w \times w}[\xi]$ nonsingular. Let $X\in\R^{n \times \w}[\xi]$ be a minimal state map for $\B$. Assume that $\B$ is asymptotically stable. Let $Q$,$Y$, and $\Psi$ satisfy (\ref{eq:PLE}), and assume that either $Q$ or $Y$ is $R$-canonical. There exist $\overline{K}=\overline{K}^\top\in\R^{n\times n}$, $\overline{Y}\in\R^{\w\times n}$, $\overline{Q}\in\R^{\bullet \times n}$ such that $\Psi(\zeta,\eta)=X(\zeta)^\top \overline{K} X(\eta)$, $Y(\xi)=\overline{Y} X(\xi)$, and $Q(\xi)=\overline{Q} X(\xi)$. Write $R(\xi)=\sum_{i=0}^L R_i\xi^i$, with $R_i\in\R^{\w\times\w}$, $i=0,\ldots,L$; then there exist $X_i\in\R^{n\times\w}$, $i=0,1,...,L-1$, such that $X(\xi)=\sum_{i=0}^{L-1} X_i \xi^i$. Moreover, denote the {coefficient matrices} of $R(\xi)$ and $X(\xi)$ by $\widetilde{R}:=\begin{bmatrix} R_0&\ldots&R_L \end{bmatrix}$ and $\widetilde{X}:=\begin{bmatrix} X_0&\ldots&X_{L-1} \end{bmatrix}$. The following statements are equivalent: 
\begin{itemize}
\item[1.] $\Psi(\zeta,\eta)$, $Y(\xi)$ and $Q(\xi)$ satisfy (\ref{eq:PLE}); 
\item[2.] There exist $\overline{K}=\overline{K}^\top\in\R^{n\times n}$, $\overline{Y}\in\R^{\w\times n}$, $\overline{Q}\in\R^{\bullet \times n}$ such that 
\begin{eqnarray}\label{eq:PLEviaLMI}
&&\begin{bmatrix} 0_{\w \times n}\\ \widetilde{X}^\top
\end{bmatrix} \overline{K} \begin{bmatrix}\widetilde{X}& 0_{n \times \w}
\end{bmatrix}+\begin{bmatrix} \widetilde{X}^\top\\0_{\w \times n}
\end{bmatrix} \overline{K} \begin{bmatrix}  0_{n \times \w}&\widetilde{X}\end{bmatrix}-\begin{bmatrix} \widetilde{X}^\top\\0_{\w \times n}
\end{bmatrix}\overline{Y}^\top \widetilde{R}- \widetilde{R}^\top  \overline{Y} \begin{bmatrix}  \widetilde{X}&0_{n \times \w}\end{bmatrix}\nonumber 
\\
&&+ \begin{bmatrix}  \widetilde{X}^\top\\0_{\w \times n}
\end{bmatrix} \overline{Q}^\top \overline{Q} \begin{bmatrix}   \widetilde{X}&0_{n \times \w}\end{bmatrix}=0\; .
\end{eqnarray}
\end{itemize}
If moreover, $\rank~\mbox{\rm col}(R(\lambda),Q(\lambda))=\tt w$ for all $\lambda\in\C$, then $1)$ is equivalent with $2)$ and $\overline{K}>0$.
\end{proposition}
\begin{IEEEproof} See Appendix \ref{app:proofs}. \end{IEEEproof}

We  now show how to compute MLFs for SLDS. For ease of exposition we assume that all roots of $\det R_k(\xi)$ have equal algebraic  and geometric multiplicity. 
\begin{theorem}\label{th:LMI4MLF}
Let $\Sigma$ be a SLDS (see Def. \ref{def:SLDS}), with $\B_k=\ker R_k\derb$ asymptotically stable, $k=1,\ldots,N$ and $R_k\in\R^{\w\times\w}[\xi]$  nonsingular. Let $X_k\in\R^{n \times \w}[\xi]$ be a minimal state map for $\B_k$. Write $R_k(\xi)=\sum_{i=0}^{L_k} R_{k,i}\xi^i$, and denote the  coefficient matrix of $R_k(\xi)$ by $\widetilde{R}_k:=\begin{bmatrix} R_{k,0}&\ldots&R_{k,L_k}\end{bmatrix}$ and that of $X_k(\xi)$ by $\widetilde{X}_k:=\begin{bmatrix}X_{k,0}&\ldots&X_{k,L_k-1}\end{bmatrix}$. Denote the roots of $\det~R_k(\xi)$ by $\lambda_{k,i}$, $i=1,\ldots n_k$. Assume that the algebraic multiplicity of $\lambda_{k,i}$ equals  the dimension of $\ker R_k(\lambda_{k,i})$. Let  $w_{k,i}\in\C^{\w}$ be such that $R_k(\lambda_{k,i})w_{k,i}=0$, with the $w_{k,i}$ associated with equal $\lambda_{k,i}$  linearly independent. Define $V_k\in\C^{n_k\times n_k}$ by $V_k:=\begin{bmatrix}X_k(\lambda_{k,1})w_{k,1}& \ldots& X_k(\lambda_{k,n_k})w_{k,n_k}\end{bmatrix}$, $k=1,\ldots,N$. Denote by $L_{k\rightarrow \ell}$, $k,\ell=1\ldots,N$, $k\neq \ell$, the re-initialisation maps of $\Sigma$. 

If there exist $\overline{K}_k\in\R^{n_k\times n_k}$, $\overline{Y}_k\in\R^{\w\times n_k}$,  $k=1\ldots,N$ such that
\begin{eqnarray}\label{eq:LMI4MLF1}
&&\widetilde{\Phi}_k:=\begin{bmatrix} 0_{\w \times n}\\ \widetilde{X}_k^\top
\end{bmatrix} \overline{K}_k\begin{bmatrix}\widetilde{X}_k& 0_{n \times \w}
\end{bmatrix}+\begin{bmatrix} \widetilde{X}_k^\top\\0_{\w \times n}
\end{bmatrix} \overline{K}_k \begin{bmatrix}  0_{n \times \w}&\widetilde{X}_k\end{bmatrix}-\begin{bmatrix} \widetilde{X}_k^\top\\0_{\w \times n}
\end{bmatrix}\overline{Y}_k^\top \widetilde{R}_k\nonumber\\&&- \widetilde{R}_k^\top  \overline{Y}_k \begin{bmatrix}  \widetilde{X}_k&0_{n \times \w}\end{bmatrix}\leq 0\; ,
\end{eqnarray}
then there exist $\overline{F}_k\in\R^{n_k\times n_k}$ such that $\widetilde{\Phi}_k=\begin{bmatrix} \widetilde{X}_k^\top\\0_{\w \times n}\end{bmatrix} \overline{F}_k\begin{bmatrix} \widetilde{X}_k&0_{n \times \w} \end{bmatrix}$, $k=1,\ldots,N$. 

Moreover, if for $k,\ell=1,\ldots,N$, $\ell \neq k$, it holds that 
\begin{eqnarray}\label{eq:LMI4MLF2}
\overline{F}_k <0 \mbox{\rm \; and }
V_k^{\ast}  \overline{K}_k V_k \geq  V_k^\ast L_{k\rightarrow \ell}^\top \overline{K}_\ell L_{k\rightarrow \ell}V_k\; ,
\end{eqnarray}
then $\Sigma$ is asymptotically stable, and $\{ X_k(\zeta)^\top \overline{K}_k X_k(\eta)\}_{k=1,\ldots,N}$ induces a MLF. 
\end{theorem}
\begin{IEEEproof}See Appendix \ref{app:proofs}. 
\end{IEEEproof}
Th. \ref{th:LMI4MLF} reduces the computation of quadratic MLFs to the solution of a system of structured LMIs (\ref{eq:LMI4MLF1})-(\ref{eq:LMI4MLF2}), a straightforward matter for standard LMI solvers. 

\begin{remark}\rm 
The fact that no (multiple) quadratic Lyapunov function exists cannot be used to conclude that a system is unstable: the class of quadratic Lyapunov functionals is not \emph{universal} in the sense of \cite{Blanchini95}, see Corollary 4.3 and Remark 4.1 p. 457. On this, see also Example \ref{ex:MLFs} below. 

 The class of \emph{polyhedral Lyapunov functions} (PLFs) is universal for linear systems with structured uncertainties;  in \cite{polyhedral} PLFs are applied  to linear switched systems in state space form, and a numerical procedure to overcome the  complexity of PLF computations is illustrated, see pp. 1021-1022  \emph{ibid}. \qed
\end{remark}

\begin{example}\label{ex:MLFs}\rm 
The SLDS in Ex. \ref{ex:elcirc} is stable. A MLF is $(Q_{\Psi_1},Q_{\Psi_2})$, where $\Psi_1(\zeta,\eta)=\begin{bmatrix} 0\\ 1\end{bmatrix} \begin{bmatrix} 0& 1\end{bmatrix}=\Psi_2(\zeta,\eta)$, inducing the QDFs $Q_{\Psi_1}(w)=w_2^2=Q_{\Psi_2}(w)$.  Their derivatives along  $\B_1$ and $\B_2$ equal $-2 w_2\frac{d}{dt} w_2=-2 w_2^2$;  due to the gluing conditions, the value of the MLF is the same before and after the switch. 

For the system in Ex. \ref{ex:math}, straightforward computations show that since the only $R_i$-canonical quadratic Lyapunov functionals for $\B_i$ are of the form $\Psi_i(\zeta,\eta)=c\begin{bmatrix} 0\\ 1\end{bmatrix} \begin{bmatrix} 0& 1\end{bmatrix}$, $i=1,2$ for $c>0$, no quadratic multiple Lyapunov functions for the SLDS exist. In fact,  an argument analogous to that of pp. 126-ff. of \cite{trennThesis} proves that  the system is unstable. \qed 
\end{example}

\begin{remark}\rm
QDFs act on $\Cinf$-functions, while  trajectories of a SLDS are  non-differentiable; however, this mismatch in differentiability is  irrelevant to Th. \ref{th:LMI4MLF} and  the other results of this paper. Indeed, we only use the calculus of QDFs as an \emph{algebraic} tool. For example, in the proof of Th. \ref{th:LMI4MLF} when considering the value of $Q_{\Psi_k}$ and $Q_{\Psi_\ell}$ before and after a switch, only the properties of their coefficient matrices are  used. \qed
\end{remark}

\begin{remark}\label{rem:LMImorecons}\rm 
Th. \ref{th:LMI4MLF} and the associated LMI-based procedure to find a MLF assume that the  $\lambda_{k,i}$ and   associated directions $w_{k,i}$ are known. If one wants to avoid such pre-computations, a weaker (i.e. more conservative) sufficient condition for the existence of a multiple Lyapunov function can be obtained by solving (\ref{eq:LMI4MLF1}) together with $\overline{F}_k<0$ and $\overline{K}_k \geq   L_{i\rightarrow \ell}^\top \overline{K}_\ell L_{k\rightarrow \ell}$ in place of (\ref{eq:LMI4MLF2}).

\end{remark}

\begin{remark}\rm 
For state-space switched systems, $R_k(\xi)=\xi I_n-A_k$ and $X_k(\xi)=I_n$, $k=1,\ldots,N$; straightforward computations yield that in (\ref{eq:LMI4MLF1}) $\overline{Y}_k=\overline{K}_k$; with  the first condition in (\ref{eq:LMI4MLF2}) we obtain the matrix Lyapunov equations $A_k^\top \overline{K}_k+\overline{K}_kA_k< 0$. The second condition in (\ref{eq:LMI4MLF2}) reduces to the classical condition on the reset maps (see e.g. Cor. 2.2 of \cite{paxman}). For the  case of switched DAE's, see sect. 6.3 of \cite{trenn}. 
\end{remark}

We conclude with an   example illustrating our modelling framework and the application of Th. \ref{th:LMI4MLF} in a realistic setting. 
\begin{example}\label{ex:sourceconv}\rm 
Some \emph{source converters} used in distributed power systems (see e.g. \cite{PeLe04}) consist of a traditional DC-DC boost converter coupled with a (dis-)connectable load, see Fig. \ref{boostRL}. 
\begin{figure}[htb!]
\begin{center}
\includegraphics[scale=0.5]{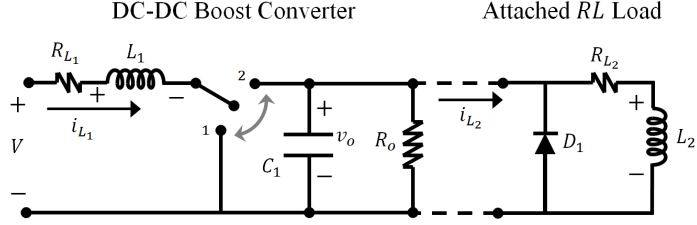}    
\caption{Source converter.}
\label{boostRL}
\end{center}
\end{figure}
We take $w=\mbox{\rm col}(i_{L_1},v_o)$ as the external variable. In order to deal with autonomous behaviours, set the input voltage $V=0$. From standard circuit modelling we conclude that the modes are $\mathcal{F}=\{\B_k=\ker~R_k\derb\}_{k=1,\ldots,4}$ where

\begin{eqnarray*}
& &R_1(\xi):=\begin{bmatrix} L_1 \xi +R_{L_1}&0\\0&C_1 \xi +\displaystyle \frac{1}{R_o}\end{bmatrix}\;,\;R_2(\xi):=\begin{bmatrix} L_1 \xi +R_{L_1}&1\\-1&C_1 \xi +\displaystyle \frac{1}{R_o}\end{bmatrix}\;, \\
& &R_3(\xi):=\begin{bmatrix} L_1 \xi +R_{L_1}&0\\0&L_2 C_1 \xi^2 + \left(R_{L_2} C_1 + \displaystyle\frac{L_2}{R_o}\right)\xi+ \displaystyle\frac{R_{L_2}}{R_o}+1\end{bmatrix}\\
& &R_4(\xi):=\begin{bmatrix} L_1 \xi +R_{L_1}&1\\-L_2 \xi -R_{L_2}&L_2 C_1 \xi^2 + \left(R_{L_2} C_1 + \displaystyle\frac{L_2}{R_o}\right)\xi+ \displaystyle\frac{R_{L_2}}{R_o}+1\end{bmatrix}\;. 
\end{eqnarray*}

The  gluing conditions derived from physical considerations are 
\begin{eqnarray*}
\left(I_2,I_2\right)&=&\left(G^+_{1\rightarrow 2}(\xi),G^-_{1\rightarrow 2}(\xi)\right)=\left(G^+_{2\rightarrow 1}(\xi),G^-_{2\rightarrow 1}(\xi)\right)=
\left(G^+_{3\rightarrow 1}(\xi),G^-_{3\rightarrow 1}(\xi)\right)\\
& =&\left(G^+_{3\rightarrow 2}(\xi),G^-_{3\rightarrow 2}(\xi)\right)=\left(G^+_{4\rightarrow 1}(\xi),G^-_{4\rightarrow 1}(\xi)\right)=\left(G^+_{4\rightarrow 2}(\xi),G^-_{4\rightarrow 2}(\xi)\right)\;;
\end{eqnarray*}
\begin{eqnarray*}
\left(G^+_{1\rightarrow 3}(\xi),G^-_{1\rightarrow 3}(\xi)\right)&:=&\left(\begin{bmatrix} 1&0\\0&1\\0 & -C_1\xi- \displaystyle \frac{1}{R_o}\end{bmatrix},\begin{bmatrix} 1& 0\\0 &1 \\0 & 0\end{bmatrix}\right)=:\left(G^+_{2\rightarrow 3}(\xi),G^-_{2\rightarrow 3}(\xi)\right)\;;
\end{eqnarray*}
\[
\left(G^+_{1\rightarrow 4}(\xi),G^-_{1\rightarrow 4}(\xi)\right):= \left(\begin{bmatrix} 1&0\\0&1\\1 & -C_1\xi- \displaystyle \frac{1}{R_o}\end{bmatrix},\begin{bmatrix} 1& 0\\0 &1 \\0 & 0\end{bmatrix}\right)=:\left(G^+_{2\rightarrow 4}(\xi),G^-_{2\rightarrow 4}(\xi)\right)
\]
\[
\left(G^+_{3\rightarrow 4}(\xi),G^-_{3\rightarrow 4}(\xi)\right):=\left(\begin{bmatrix} 1&0\\0&1\\1 & -C_1\xi- \displaystyle \frac{1}{R_o}\end{bmatrix},\begin{bmatrix} 1&0\\0&1\\0 & -C_1\xi- \displaystyle \frac{1}{R_o}\end{bmatrix}\right)=:\left(G^-_{4\rightarrow 3}(\xi),G^+_{4\rightarrow 3}(\xi)\right)\;. 
\] 
The following polynomial differential operators  induce state maps for $\B_k$, $k=1,\ldots,4$: 
\[
X_1(\xi)=X_2(\xi):=\begin{bmatrix} 1&0\\0&1\end{bmatrix}\;;\;\;X_3(\xi):=\begin{bmatrix} 1&0\\0&1\\0 & -C_1\xi- \displaystyle \frac{1}{R_o}\end{bmatrix}\;;\;\;X_4(\xi):=\begin{bmatrix} 1&0\\0&1\\1 & -C_1\xi- \displaystyle \frac{1}{R_o}\end{bmatrix}\;.
\]
They can be derived by physical considerations or automatically, using the procedures in \cite{statemaps}. Proceeding as in sect. \ref{sec:wellpos}, we compute the  re-initialisation maps 
\[
L_{1\rightarrow 2}=L_{2\rightarrow 1}:=\begin{bmatrix} 1&0\\0&1\end{bmatrix}\;;\;\;L_{1\rightarrow 3}=L_{1\rightarrow 4}=L_{2\rightarrow 3}=L_{2\rightarrow 4}:=\begin{bmatrix} 1&0\\0&1\\0&0\end{bmatrix}\;;
\]
\[
L_{3\rightarrow 4}=L_{4\rightarrow 3}:=I_3\;;\;\;L_{3\rightarrow 1}=L_{3\rightarrow 2}=L_{4\rightarrow 1}=L_{4\rightarrow 2}:=\begin{bmatrix} 1&0&0\\0&1&0\end{bmatrix}\;.
\]
With the parameters $L_1 = 100 \mu F$, $R_{L_1} = 0.01 \Omega$, $C_1 =100 \mu F$, $R_o = 2\Omega$, $R_{L_2} = 0.02\Omega$, $L_2 = 100\mu F$ we obtain the characteristic frequencies $\lambda_{1,1}=-5000$, $\lambda_{1,2}=-100$, $\lambda_{2,1}= -2550 +     j9695.2=\overline{\lambda_{2,2}}$, $\lambda_{3,1}= -2600 +     j9707.7=\overline{\lambda_{3,2}}$, $\lambda_{3,3}=-100$, $\lambda_{4,1}=-149.94 $, $\lambda_{4,2}= -2575 +     j 13933=\overline{\lambda_{4,3}}$. The $V$-matrices of Th. \ref{th:LMI4MLF}  are 
\[
V_1=\begin{bmatrix} 0 & 1 \\ 1 & 0 \end{bmatrix}\;;\;\;V_2=\begin{bmatrix} \displaystyle 0.70711   &  0.70711      \\  0.17324 -    j 0.68556  & \displaystyle 0.17324 +   j 0.68556 \end{bmatrix}\;,
\]
\[
V_3=\begin{bmatrix} 0 & 0 & 1 \\  0.16971 +    j 0.68644    &  0.16971 -    j 0.68644 &0 \\ 0.62564 -   j 0.32949&  0.62564 +   j 0.32949 & 0 \end{bmatrix}\;;
\]
\[
V_4=\begin{bmatrix} 0.70796 & 0.08739 +    j 0.49199  &0.08739 -    j 0.49199  \\  0.00353    & 0.70711   & 0.70711   \\ 0.70625   &   -0.08407 -    j 0.49323 &  -0.17147 +    j 0.98522 \end{bmatrix}\;.
\]
Using standard LMI solvers for the LMIs (\ref{eq:LMI4MLF1}), (\ref{eq:LMI4MLF2}) we obtain
\[
\overline{K}_1=\overline{K}_2=\begin{bmatrix}   0.00123 &    -0.00002 \\
    -0.00002  &     0.00112  \end{bmatrix}\;;\;\;\overline{K}_3=\overline{K}_4=\begin{bmatrix}  0.00123 &    -0.00002 & 0\\
    -0.00002  &      0.00112  & 0 \\ 0 & 0 &0.00121 \end{bmatrix}\;. 
\]
Applying Th. \ref{th:LMI4MLF} we conclude that $\{X_k(\zeta)^\top K_k X_k(\eta)\}_{k=1,\ldots,4}$ induces a MLF. 

To illustrate the modularity of our modelling framework, assume that the source converter can also be connected to yet another $RC$ load as depicted in Fig. \ref{boostRC}. 
\begin{figure}[htb!]
\begin{center}
\includegraphics[scale=0.5]{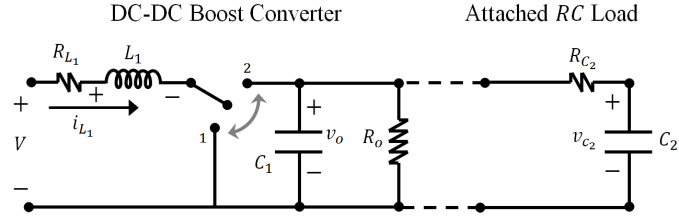}    
\caption{DC-DC Boost converter/$RC$-Circuit interconection.}
\label{boostRC}
\end{center}
\end{figure}
This results in two additional behaviours in $\mathcal{F}$, namely $\B_i=\ker~R_i\derb$, $i=5,6$, where 
\begin{eqnarray*}
R_5(\xi)&:=&\begin{bmatrix} L_1 \xi +R_{L_1}&0\\0&R_{C_2}C_1 C_2 \xi^2 +\left(\displaystyle\frac{R_{C_2}C_2}{R_o}+C_1+C_2\right)\xi+\displaystyle\frac{1}{R_o}\end{bmatrix}\\
R_6(\xi)&:=&\begin{bmatrix} L_1 \xi +R_{L_1}&1\\-R_{C_2}C_2 \xi - 1&R_{C_2}C_1 C_2 \xi^2 +\left(\displaystyle\frac{R_{C_2}C_2}{R_o}+C_1+C_2\right)\xi+\displaystyle\frac{1}{R_o}\end{bmatrix}\;.
\end{eqnarray*}
We choose as state maps for $\B_5$ and $\B_6$ 
\[
X_5(\xi):=\begin{bmatrix} 1&0\\0&1\\0 & R_{C_2}C_1 \xi + \displaystyle\frac{R_{C_2}}{R_o}+1\end{bmatrix}\;;\;\;X_6(\xi):=\begin{bmatrix} 1&0\\0&1\\-R_{C_2} & R_{C_2}C_1 \xi + \displaystyle\frac{R_{C_2}}{R_o}+1\end{bmatrix}\;, 
\]
corresponding to the re-initialisation maps
\[
L_{5\rightarrow 6}=L_{6\rightarrow 5}:=I_3\;;\;\;L_{1\rightarrow 5}=L_{1\rightarrow 6}=L_{2\rightarrow 5}=L_{2\rightarrow 6}:=\begin{bmatrix} 1&0\\0&1\\0&0\end{bmatrix}\;;
\]
\[
L_{5\rightarrow 1}=L_{5\rightarrow 2}=L_{6\rightarrow 1}=L_{6\rightarrow 2}:=\begin{bmatrix} 1&0&0\\0&1&0\end{bmatrix}\;.
\]
Given the values $R_{C_2}=1 \Omega$, $C_2=100 \mu F$, in order to compute a MLF for $\mathcal{F}:=\{\B_k\}_{k=1,\ldots,6}$ we only need to add two LMIs to those set up previously; the solution is
\[
\overline{K}_1=\overline{K}_2=\begin{bmatrix}  0.00127 &    -0.00002 \\
    -0.00002  &     0.00126  \end{bmatrix}\;;\;\;\overline{K}_3=\overline{K}_4=\begin{bmatrix}  0.00127 &    -0.00002 & 0\\
    -0.00002  &      0.00126  & 0 \\ 0 & 0 &0.00131 \end{bmatrix}\;;
\]
\[
\overline{K}_5=\overline{K}_6=\begin{bmatrix}   0.00127  &    -0.00002 & 0\\
    -0.00002 &      0.00126  & 0 \\ 0 & 0 &0.00382 \end{bmatrix}\;.
\]
\qed
\end{example}

\section{Positive-realness and Stability of Standard SLDS}\label{sec:standard}

The PLE's  resemblance to the dissipation equality (see App. \ref{sec:Dissip}) underlies the results of this section, aimed at connecting positive-realness and  stability of two-modes SLDS (see  \cite{shorten,circle}  in the classical setting).  We begin by recalling the definition of strict positive-real rational function (note that this definition  is not universally accepted; cf. \cite{tao}, Th. 2.1.). 
\begin{definition}\label{Def:PRM}\rm
$G\in\R^{\tt w \times w}(\xi)$ is \emph{strictly positive-real} if it is analytic in $\C_+$ and $G(-j \omega)^\top + G(j \omega) > 0 \,\, \forall \, \omega \, \in \, \mathbb{R}$. 
\end{definition}
We now relate  the PLE (\ref{eq:PLE}) with strict positive-realness of an associated transfer function.  
\begin{proposition}\label{prop:SPR&PLE}
Let $N,D\in\R^{\w\times\w}[\xi]$. Assume that $D$ and $N$ are Hurwitz, and that $ND^{-1}$ is strictly proper and strictly positive real. There exist $Q\in\R^{\bullet \times \w}[\xi]$ such that $D(-\xi)^\top N(\xi)+N(-\xi)^\top D(\xi)=Q(-\xi)^\top Q(\xi)$; moreover $\rank~\col(D(\lambda),Q(\lambda))=\w$ for all $\lambda\in\C$, and $QD^{-1}$ is strictly proper. Define $\Psi(\zeta,\eta):=\frac{D(\zeta)^\top N(\eta)+N(\zeta)^\top D(\eta)-Q(\zeta)^\top Q(\eta)}{\zeta+\eta}$. Then $\Psi(\zeta,\eta)$ is a $D$-canonical Lyapunov function for $\ker D\derb$, and $\Psi(\zeta,\eta) \mod N$ is a Lyapunov function for $\ker N\derb$. 
\end{proposition}
\begin{proof}
See Appendix \ref{app:proofs}.
\end{proof}


Thus if $\Psi$ is a suitable storage function of the system with transfer function $ND^{-1}$, associated with a supply rate induced by $\begin{bmatrix} 0 & I \\ I & 0  \end{bmatrix}$ and with dissipation rate $Q(\zeta)^\top Q(\eta)$, then it is also a Lyapunov function for $\ker D\derb$ and (after the ``mod" operation) also for $\ker~N\derb$ (on  dissipativity and Lyapunov stability see also \cite{RK}). Remarkably, it turns out that such storage functions also induce a \emph{MLF} for a SLDS with modes $\ker N\derb$, $\ker D\derb$, and special gluing conditions, naturally associated with the ``$\mbox{\rm mod}$" operation. We now define such systems.

%

In the following, we consider SLDSs where ${\mathcal F} = \left(\ker R_1\left(\frac{d}{dt}\right), \ker R_2\left(\frac{d}{dt}\right) \right)$, with $R_j\in\R^{\tt w \times \tt w}[\xi]$, $j=1,2$ nonsingular. We assume that \emph{$R_2 R_1^{-1}$ is strictly proper}; this implies that the state space of $\ker R_2\derb$ is included in that of $\ker~R_1\derb$, as we presently show. 
\begin{lemma}\label{lemma:state}
Let $\mathfrak{B}_i=\ker R_i\left(\frac{d}{dt}\right)$, $i=1,2$. Assume that $R_1,R_2\in\R^{\tt w \times \tt w}[\xi]$ are nonsingular, and that $R_2 R_1^{-1}$ is strictly proper. Let $n_i:=\deg(\det(R_i))$; then $n_2<n_1$. There exist $X_1^\prime\in\R^{(n_1-n_2)\times\tt w}[\xi]$, $X_2\in\R^{n_2\times\tt w}[\xi]$ such that $X_2\derb$ is a  minimal state map for $\mathfrak{B}_2$, and 
\begin{equation}\label{stateB1}
X_1\derb:= \mbox{\rm col}\left(X_2\derb,X_1^\prime\derb\right)\; ,
\end{equation}
is a  minimal state map  for $\mathfrak{B}_1$. Moreover, there exists $\Pi\in\R^{(n_1-n_2)\times n_2}$ such that $X_1^\prime(\xi)~\mbox{\rm mod}~ R_2=\Pi X_2(\xi)$.
\end{lemma}
\begin{IEEEproof}See Appendix \ref{app:proofs}.
\end{IEEEproof}
\begin{example}\label{ex:scalarcase}\rm
If $\w=1$, $R_2R_1^{-1}$ is strictly proper iff $n_1=\deg(R_1)>\deg(R_2)=n_2$.  A  state map for $\B_1$ is $\col(\xi^k)_{k=0,\ldots,n_1-1}$, whose first $n_2$ elements form a  basis for the state space of  $\B_2$. The rows of $\Pi$ consist of the coefficients of the polynomials $\xi^k~ \mbox{\rm mod}~ R_2(\xi)$, $k=n_2,\ldots,n_1-1$. \qed
\end{example}
In the rest of this section we consider \emph{standard SLDS}, defined as follows. 
\begin{definition}\label{def:standardSLDS}\rm
Let $\Sigma=\{\mathcal{P},\mathcal{F},\mathcal{S},\mathcal{G}\}$  be a SLDS with ${\mathcal F} =  \left(\ker R_1\left(\frac{d}{dt}\right), \ker R_2\left(\frac{d}{dt}\right) \right)$, where  $R_j\in\R^{\tt w \times \tt w}[\xi]$ is nonsingular, $j=1,2$. Assume that  $R_2 R_1^{-1}$ is strictly proper. Let $n_j:=\mbox{\rm deg}(\mbox{\rm det}(R_j))$, $j=1,2$, and let $X_1^\prime\in\R^{(n_1-n_2)\times\tt w}[\xi]$, $X_2\in\R^{n_2\times\tt w}[\xi]$ and $\Pi\in\R^{(n_1-n_2)\times n_2}$ be as in Lemma \ref{lemma:state}. $\Sigma$ is a \emph{standard SLDS} if the gluing conditions are $\left({G}_{2 \rightarrow 1}^-(\xi),{G}_{2 \rightarrow 1}^+(\xi)\right):=\left(\col(
X_2(\xi),
\Pi X_2(\xi)), \col(
X_2(\xi),
X_1^\prime(\xi))
\right)$ and $\left({G}_{1 \rightarrow 2}^-(\xi),{G}_{1 \rightarrow 2}^+(\xi)\right):=\left(X_2(\xi), X_2(\xi)\right)$. 
\end{definition}
It is straightforward  to check that the gluing conditions of a standard SLDS are well-posed.   
\begin{remark}\rm 
In a standard SLDS, the state space of $\B_2$ is contained in that of $\B_1$; however, at any time the state used for the description of the system is that of the active dynamics, and not a global one.
\end{remark}
\begin{example}\rm
 Assume that $R_1$ and $R_2$ in Ex. \ref{ex:scalarcase} are monic, and  that $n_1=n_2+1$. Denote $R_2(\xi)=:\sum_{j=0}^{n_1-1} R_{2,j} \xi^j$, and define $S(\xi):=\begin{bmatrix} 1&\ldots&\xi^{n_1-2} \end{bmatrix}^\top$. The gluing conditions of the standard SLDS are $\left(G_{2\rightarrow 1}^-(\xi), G_{2\rightarrow 1}^+(\xi)\right)=\left(\col(S(\xi), -\sum_{j=0}^{n_1-2} R_{2,j} \xi^j  ),\col(S(\xi),\xi^{n_1-1})  \right)$ and $\left(G_{1\rightarrow 2}^-(\xi), G_{1\rightarrow 2}^+(\xi)\right)=
  \left(S(\xi),S(\xi)\right)$. In a switch $\B_2\rightarrow\B_1$, to obtain ``initial conditions''  uniquely specifying $w\in\B_1$, we need to define the  value of $\frac{d^{n_1-1}}{dt^{n_1-1}}w$ after the switch. Standard gluing conditions stipulate that it coincides with $\frac{d^{n_1-1}}{dt^{n_1-1}}w=-\sum_{i=0}^{n_1-2} R_{2,i}\frac{d^{i}}{dt^{i}}w$, since before the switch $w\in\B_2$. In a switch  $\B_1\rightarrow\B_2$, we project the vector of derivatives characteristic of  $w\in\B_1$ down onto the shorter vector of derivatives of $w\in\B_2$.   \qed 
\end{example}

We now prove that a standard SLDS where $R_2R_1^{-1}$ is strictly positive real admits a multiple Lyapunov function induced by $\{\Psi_1,\Psi_2\}$ where $\Psi_1$ is a storage function for $R_2R_1^{-1}$, and $\Psi_2=\Psi_1 \mbox{\rm ~mod}~ R_2$. This is the multivariable generalisation of some results presented in \cite{cdcpaper}, \cite{mtns12}. 
\begin{theorem}\label{th_01}
Let $\Sigma$ be a standard SLDS (see Def. \ref{def:standardSLDS}), with $R_1$ and $R_2$ Hurwitz. Assume that $R_2R_1^{-1}$ is  strictly proper and strictly positive-real. Define $\Phi(\zeta,\eta):= R_1(\zeta)^\top R_2(\eta) + R_2(\zeta)^\top R_1(\eta)$. There exists $Q\in\R^{\bullet \times \w}[\xi]$ such that $\Phi(-\xi,\xi)=Q(-\xi)^\top Q(\xi)$, $\mbox{\rm rank}~\col(R_1(\lambda),Q(\lambda))=\w$ for all $\lambda\in\C$ and $QR_1^{-1}$ is strictly proper. Define 
\begin{equation}\label{polystorage2}
\Psi_1(\zeta,\eta):=\!\frac{\Phi(\zeta,\eta) -Q(\zeta)^\top Q(\eta)}{\zeta + \eta}\;.
\end{equation}
Then $\Psi_1$ is $R_1$-canonical. Moreover, define $\Psi_2:=\Psi_1~ \mbox{\rm mod}~ R_2$; then  $\{\Psi_1, \Psi_2\}$ induces a multiple Lyapunov function for $\Sigma$.
\end{theorem}

\begin{IEEEproof}
See Appendix \ref{app:proofs}.
\end{IEEEproof}

Th. \ref{th_01} yields two approaches to computing a MLF for a standard SLDS. The first  is algebraic and consists of a  polynomial spectral factorisation and the computation of $\Psi_1$ from   (\ref{polystorage2}). The second, based on LMIs,  arises from the proof of Th. \ref{th_01}. We state it in the following result. 
\begin{corollary}
Let $X(\xi)$ be a minimal state map for $\B_1$ as in Lemma \ref{lemma:state}, and denote by $\widetilde{R}_1$ the coefficient matrix of $R_1(\xi)$.  Under the assumptions of Th. \ref{th_01}, there exist $\overline{Y}\in\R^{\w \times n_1}$, $\Psi_{11}\in\R^{n_2 \times n_2}$, and $\Psi_{22}\in\R^{(n_1-n_2) \times (n_1-n_2)}$ such that $\widetilde{\Psi}_1:=
\begin{bmatrix}
\Psi_{11} & -\Pi^\top \Psi_{22}\\
 -\Psi_{22} \Pi & \Psi_{22}
\end{bmatrix}>0$ satisfies the LMI
\[
\begin{bmatrix} 0_{\w \times n_1}\\ \widetilde{X}^\top
\end{bmatrix} \widetilde{\Psi}_1\begin{bmatrix}\widetilde{X}& 0_{n_1 \times \w}
\end{bmatrix}+\begin{bmatrix} \widetilde{X}^\top\\0_{\w \times n_1}
\end{bmatrix} \widetilde{\Psi}_1 \begin{bmatrix}  0_{n \times \w}&\widetilde{X}\end{bmatrix}-\begin{bmatrix} \widetilde{X}^\top\\0_{\w \times n_1}
\end{bmatrix}\overline{Y}^\top \widetilde{R}_1- \widetilde{R}_1^\top  \overline{Y} \begin{bmatrix}  \widetilde{X}&0_{n_1 \times \w}\end{bmatrix}\leq 0\; . 
\]
Then $X(\zeta)^\top \widetilde{\Psi}_1 X(\eta)$ and $X_2(\zeta)^\top \left(\Psi_{11}- \Pi^\top\Psi_{22}\Pi\right)X_2(\eta)$ induce a MLF for $\Sigma$. 
\end{corollary}

\begin{remark}\rm
If $\w=1$ the proof of Th. \ref{th_01} simplifies considerably; see  \cite{mtns12} for details. \qed \end{remark}

\begin{remark}\rm 
Theorem \ref{th_01} holds  also if $R_2 R_1^{-1}$ is \emph{bi-proper}, i.e. proper and with a proper inverse; note that in this case the state spaces of $\B_1$ and of $\B_2$ coincide. Let $X\in\R^{\bullet \times \bullet}[\xi]$ be a state map for $\B_1$; the standard gluing conditions are $(G_{1 \rightarrow 2}^-(\xi),G_{1 \rightarrow 2}^+(\xi))=(X(\xi),X(\xi))=(G_{2 \rightarrow 1}^-(\xi),G_{2 \rightarrow 1}^+(\xi))$. It is  straightforward  to check that e.g. the largest storage function for $R_2 R_1^{-1}$  yields a  MLF. For $\w=1$ this  is shown in \cite{cdcpaper}. \qed \end{remark}

\begin{remark}\rm 
In the state-space framework it is well-known that if the open-loop transfer function of a system is positive-real, then all closed-loop systems obtained from it  by state feedback share a  common quadratic Lyapunov function (see  sect. 2.3.2 of \cite{liberzon} and \cite{shorten,circle}). Th. \ref{th_01} offers a new perspective on the relation between positive-realness and stability: in our framework, the different dynamical regimes do not arise from closing the loop around some fixed plant, and  positive-realness arises from  the interplay of the mode dynamics. \qed \end{remark}

\begin{remark}\rm 
Theorem \ref{th_01}  can also be used to compute from a given  Hurwitz matrix $R_1$, some  matrix $R_2$ such that the SLDS with  modes $\ker~R_i\derb$, $i=1,2$ and standard gluing conditions is asymptotically stable. Namely, select $Q\in\mathbb{R}^{\bullet \times \w}[\xi]$ such that  $\mbox{\rm rank}~\col(R_1(\lambda),Q(\lambda))=\w$ for all $\lambda\in\mathbb{C}$ and $QR_1^{-1}$ is strictly proper; solve the PLE for $R_2$. Then the standard SLDS with behaviours $\ker~R_i\derb$, $i=1,2$ is stable.  \qed 
\end{remark}
Finally, we discuss the concept of positive-real completion, defined as follows. 
\begin{definition}\label{def_comp}\rm
Let $R_i\in\R^{\w\times\w}[\xi]$, $i=1,2$ be nonsingular and $R_2 R_1^{-1}$ strictly proper. $M\in\R^{\tt w \times w}[\xi]$ is a \emph{strictly positive-real completion} of $R_2R_1^{-1}$ if $MR_2R_1^{-1}$ is strictly positive-real.
\end{definition}


The following result shows that if a MLF exists, then a positive-real completion can be found. 
\begin{theorem}\label{th_comp}
Let $\Sigma$  be a standard SLDS (see Def. \ref{def:standardSLDS}). 
If $\{\Psi_1,\Psi_1\mod R_2\}$ induces a MLF for $\Sigma$ such that $(\zeta+\eta)\Psi_1(\zeta,\eta)~\mbox{\rm mod}~ R_1=-Q(\zeta)^\top Q(\eta)$ with $\rank~Q(j\omega)=\w$ for all $\omega\in\R$ and $QR_1^{-1}$ strictly proper, then there exists a strictly positive-real completion $M\in\R^{\tt w \times w}[\xi]$ for $R_2 R_1^{-1}$.
\end{theorem}
\begin{IEEEproof}
See Appendix.
\end{IEEEproof}

\begin{remark}\rm
An interesting question is whether given a positive-real completion, a MLF induced by $\Psi_1$ and $\Psi_1~\mbox{\rm mod}~ R_2$ can be found for some $\Psi_1\in\R_s^{\tt w \times w}[\zeta,\eta]$. The existence of such a MLF can be checked by solving a structured LMI, namely that derived from the positive-real lemma for $MR_2R_1^{-1}$, together with the structural requirement that the storage function does not increase at the switching instants (see Lemma \ref{lemma:fundstructure}). Such a convex feasibility problem is analogous to those arising in structured Lyapunov problems (see \cite{BY}), and can be solved using standard LMI solvers. 
\end{remark}


\section{Conclusions}\label{sec:concl}
We  presented a framework for the modelling and stability analysis of close linear switched systems in which the  dynamical modes are not described in state-space form, and do not share a common state space. Pivotal in our approach is the concept of gluing conditions, that impose concatenation constraints on the system trajectories at the switching instants. We devised Lyapunov conditions for general gluing conditions and an arbitrary finite number of modes, amenable to be checked via systems of LMIs. We have also given Lyapunov conditions of a more algebraic flavour  based on the concept of positive-realness for two-mode SLDS. 

\appendices

\section{Background material}\label{app:backmater}
\subsection{Notation}

The space of $\n$ dimensional real
vectors is denoted by $\mathbb{R}^{\n}$, and that of $\m\times \n$ real matrices by $\mathbb{R}^{\m\times \n}$.  $\mathbb{R}^{\bullet\times \m}$ denotes the space of real matrices with $\m$ columns and an unspecified finite number of rows. 
Given  matrices $A,B\in\mathbb{R}^{\bullet\times \m}$, $\col(A,B)$ denotes the matrix obtained by stacking $A$ over $B$. 
 The ring of polynomials with real coefficients in the indeterminate $\xi$ is denoted by $\mathbb{R}[\xi]$; the ring of two-variable polynomials with real coefficients in the indeterminates $\zeta$ and $\eta$ is denoted by $\mathbb{R}[\zeta,\eta]$. $\mathbb{R}^{{\tt r}\times {\w} }[\xi]$ denotes the set of all ${\tt r}\times {\w}$ matrices with entries in $\xi$, and $\R^{\n\times\m}[\zeta,\eta]$ that of $\n\times\m$ polynomial matrices in $\zeta$ and $\eta$. The set of rational $\m\times\n$ matrices is denoted by $\R^{\m\times \n}(\xi)$. We denote by $\bar{\lambda}$  the conjugate of  $\lambda\in\C$.
The set of
infinitely differentiable functions from $\mathbb{R}$ to
$\mathbb{R}^{{\tt w}}$ is denoted by
$\mathfrak{C}^{\infty}(\mathbb{R},\mathbb{R}^{\tt w})$. If $f:\R\rightarrow \R^\bullet$, we set $f(t^-):=\lim_{\tau \nearrow t}f(\tau)$ and $f(t^+):=\lim_{\tau\searrow t}f(\tau)$ provided that these limits exist.  

\subsection{Linear differential behaviours}\label{sec:behbasics}
$\B\subseteq \mathfrak{C}^\infty (\R,\R^\w)$ is a \emph{linear time-invariant differential behaviour} if it is the set of solutions of a finite system of constant-coefficient linear differential equations, i.e. if there exists $R\in \R^{{\tt g} \times \w}[\xi]$ such that $\B = \{ w \in \mathfrak{C}^\infty (\R,\R^\w) ~|~ R(\ddt)w=0 \} =: \ker \ R(\ddt)$. If $\B=\ker R(\ddt)$, then we call $R$ a \emph{kernel representation} of $\B$. We denote with $\Lw$ the set of all linear time-invariant differential behaviours with $\w$ variables. $\B$ is  \emph{autonomous} if there are no free components in its trajectories; it can be shown that such $\B$ admits a kernel representation with $R\in\R^{\tt w \times w}[\xi]$  square and nonsingular (see \cite{yellow}, Theorem 3.2.16). 

Let $R\in\R^{\w\times\w}[\xi]$ be nonsingular, and let $f\in\R^{1\times {\tt w}}[\xi]$; $f$ is uniquely written as $fR^{-1}=s +n$, where $s\in\R^{1\times{\tt w}}(\xi)$ is a vector of strictly proper rational functions, and $n\in\R^{1\times{\tt w}}[\xi]$. We call $sR\in\R^{1\times \w}[\xi]$ the \emph{canonical representative of $f$ modulo $R$}, denoted by $f\mod R$. Note that the polynomial differential operators $f\derb$ and $f^\prime\derb$, with $f^\prime=f\mod R$, are \emph{equivalent along $\ker R\derb$} in the sense that $f\derb w=f^\prime \derb w$ for all $w\in \ker R\derb$. The definition of $R$-canonical representative extends in a natural way to polynomial matrices.

\subsection{State maps}\label{app:backmater:statemaps}
A latent variable $\ell$ (see \cite{yellow}, def. 1.3.4 ) is a \emph{state variable} for $\mathfrak{B}$ iff there exist $E,F\in\R^{\bullet \times \bullet}$, $G\in\R^{\bullet \times \w}$ such that $\mathfrak{B}=\left\{w \mid \exists~\ell \mbox{ s.t. } E\frac{d \ell}{dt}+F\ell+Gw=0\right\}$, i.e. if $\mathfrak{B}$ has a representation of first order in $\ell$ and zeroth order in $w$. The minimal number of state variables needed to represent $\mathfrak{B}$ in this way is called the \emph{McMillan degree} of $\mathfrak{B}$, denoted by $\n(\mathfrak{B})$.

A state variable for $\mathfrak{B}$ can be computed as the image of a polynomial
differential operator called a \emph{state map} (see  \cite{statemaps}). To construct state maps for  $\mathfrak{B}:=\ker R\derb$, with $R\in\R^{\w\times\w}[\xi]$ nonsingular, consider the set $\mathfrak{X}(R):=\{ f\in\mathbb{R}^{1\times \tt w}[\xi] ~\mid~ fR^{-1} \mbox{ is strictly  proper}\}$. $\mathfrak{X}(R)$ is a finite-dimensional subspace of $\mathbb{R}^{1\times \w}[\xi]$ over $\mathbb{R}$, (see \cite{statemaps}, Prop. 8.4), of dimension $n:=\deg(\det(R))$ (see \cite{statemaps}, Cor. 6.7). To compute a state map  for $\mathfrak{B}$, choose a set of generators $x_i\in\R^{1\times \w}[\xi]$, $i=1,\ldots,N$ of $\mathfrak{X}(R)$, and define $X:=\mbox{col}(x_i)_{i=1,\ldots,N}$; to obtain a \emph{minimal state map}, choose $\{x_i\}_{i=1,\ldots,N}$ so that they form a basis of $\mathfrak{X}(R)$. It can be shown that there exists a state map $X$ and  $A\in\R^{\bullet \times \bullet}$, $B\in\R^{\bullet\times\w}$ such that $\xi X(\xi)=AX(\xi)+B R(\xi)$ (see \cite{statemaps}, Th. 6.2).

Let $\B\in\Lw$, and $X\in\R^{\bullet \times \w}[\xi]$ be a state map for $\B$. A polynomial differential operator $d\left(\frac{d}{dt} \right)$ is a (linear) \emph{function of the state} of $\mathfrak{B}$ if there exists a constant vector $f\in \R^{1\times \tt{w}}$ such that $d\left(\frac{d}{dt} \right)w=fX\derb w$ for all $w\in\B$.

\subsection{Quadratic differential forms}\label{app:backmater:QDF}


Let $\Phi\in\mathbb{R}^{\w\times \w}[\zeta,\eta]$; then
$\Phi(\zeta,\eta)=\sum_{h,k}\Phi_{h,k}\zeta^{h}\eta^{k}$,
where $\Phi_{h,k}\in\mathbb{R}^{\w\times \w}$ and the sum extends over a finite set of nonnegative indices. $\Phi(\zeta,\eta)$
induces the {\em quadratic differential form} (QDF) acting on $\Cinf$-trajectories defined by $Q_{\Phi}(w):=\sum_{h,k}(\frac{d^{h}w}{dt^{h}})^{\top}\Phi_{h,k}\frac{d^{k}w}{dt^{k}}$. Without loss of generality QDF is induced by a {\em symmetric} two-variable polynomial matrix $\Phi(\zeta, \eta)$, i.e. one such that $\Phi(\zeta,\eta)=\Phi(\eta,\zeta)^\top$; we denote the set of such matrices by $\R_s^{\w\times\w}[\zeta,\eta]$. 


Given $Q_{\Psi}$, its
{\em derivative} is the QDF $Q_{\Phi}$ defined by
$Q_{\Phi}(w):= \frac{d}{dt} (Q_{\Psi}(w))$ for all
$w\in\Cinf(\R,\R^\w)$; this holds if and only if $\Phi(\zeta,\eta)=
(\zeta+\eta)\Psi(\zeta,\eta)$ (see \cite{QDF}, p. 1710). 

$Q_{\Phi}$ is \emph{nonnegative along} $\B\in\Lw$, denoted by $Q_{\Phi} \overset{\B}{\geq} 0$ if $Q_{\Phi}(w) \geq 0$ for all $w \in \B$; and \emph{positive along $\B$}, denoted by $Q_{\Phi}\overset{\B}{>}0$, if $Q_{\Phi} \overset{\B}{\geq} 0$ and $[Q_{\Phi}(w)=0~\forall w \in \B]$ $\Longrightarrow$ $[w=0]$. If $\B=\Cinf(\R,\R^\w)$, then we call $Q_\Phi$ simply \emph{nonnegative}, respectively \emph{positive}. For algebraic characterizations of these properties  see \cite{QDF}, pp. 1712-1713.

Let $R\in\R^{\tt w \times w}[\xi]$ be nonsingular and $\Phi\in\R^{\tt w \times w}[\zeta,\eta]$. Factorise $\Phi(\zeta,\eta)=M(\zeta)^\top N(\eta)$ and compute the $R$-canonical representatives (see App. \ref{sec:behbasics})   $M^\prime=M\;\mbox{mod}\;R$ and $N^\prime=N\;\mbox{mod}\;R$. The \emph{$R$-canonical representative of $\Phi(\zeta,\eta)$} is defined by $\Phi(\zeta,\eta)\;\mbox{mod}\;R:=M^\prime(\zeta)^\top N^\prime(\eta)$. The QDFs $Q_{\Phi}$, $Q_{\Phi'}$ are \emph{equivalent along $\ker R\derb$}, i.e. $Q_{\Phi'}(w)=Q_{\Phi}(w)$ for all $w\in \ker R\derb$.

\subsection{Dissipativity}\label{sec:Dissip}

A controllable (see Ch. 5 of \cite{yellow}) behaviour $\B\in\Lw$ is  \emph{dissipative} with respect to the \emph{supply rate} $Q_\Phi$ if there exists a QDF $Q_\Psi$, called a \emph{storage function}, such that $Q_\Phi(w)-\frac{d}{dt} Q_\Psi(w)\geq 0 \mbox{ for all } w\in\B$. This inequality holds iff there exists a \emph{dissipation function}, i.e. a QDF $Q_\Delta\overset{\B}{\geq} 0$ such that for all $w\in\B$ of compact support it holds that $\int_{-\infty}^{+\infty} Q_\Phi(w)(t) dt=\int_{-\infty}^{+\infty} Q_\Delta(w)(t) dt$ (see Prop. 5.4 of \cite{QDF}). Moreover, there is a one-one correspondence between storage- and dissipation functions, defined by $\frac{d}{dt} Q_\Psi(w)+Q_\Delta(w)=Q_\Phi(w)$ for all $w\in\B$. If $\B=\Cinf(\R,\R^\w)$, this equality holds if and only if $(\zeta+\eta)\Psi(\zeta,\eta)+\Delta(\zeta,\eta)=\Phi(\zeta,\eta)$.

\normalsize

\section{Proofs}\label{app:proofs}

\begin{IEEEproof}[\bf Proof of Th. \ref{th:MLFSLDS}]
Let $s\in\mathcal{S}$ be a switching signal, and  from $\{Q_{\Psi_1},\ldots,Q_{\Psi_N} \}$ define the ``switched functional" $Q_\Lambda$ acting on $\B^\Sigma$ by $Q_\Lambda(w)(t):=Q_{\Psi_{s(t)}}(w)(t)$. Observe that in every interval $[t_{j-1},t_{j})$ $Q_\Lambda$ is nonnegative, continuous and strictly decreasing, since $Q_{\Psi_{s(t_{j-1})}}$ satisfies conditions $1)-2)$. Moreover, for every admissible trajectory the value of $Q_\Lambda$ does not increase at switching instants (condition $3)$).  It follows from standard arguments (see e.g. Th. 4.1 of \cite{YMH}) that $\Sigma$ is asymptotically stable.
\end{IEEEproof}

\begin{IEEEproof}[\bf Proof of Prop. \ref{prop:PLEviaLMI}]
The existence of $\overline{K}=\overline{K}^\top\in\R^{n\times n}$, $\overline{Y}\in\R^{\w\times n}$, $\overline{Q}\in\R^{\bullet \times n}$ follows from Th. \ref{th:LyapfromQDF} and the fact that the rows of $X(\xi)$ are a basis for the vector space over $\R$ defined by $\{f\in\R^{1\times \w}[\xi] \mid fR^{-1} \mbox{ \rm is strictly proper}\}$. The fact that the degree of $X(\xi)$ is less than that of $R(\xi)$ follows from $XR^{-1}$ being strictly proper and Lemma 6.3-10 of \cite{kailath}. 

In order to prove the equivalence of statements $1.$ and $2.$, 
define $S_L(\xi):= \begin{bmatrix} I_\w& \xi I_\w& \ldots & \xi^{L} I_\w\end{bmatrix}^\top$; the equivalence follows in a straightforward way from the first part of the claim and  the equalities  $
X(\xi)=\begin{bmatrix}  X_0 &\ldots &X_{L-1}&0_{n \times \w}\end{bmatrix} S_L(\xi)$, $ 
\xi X(\xi)=\begin{bmatrix} 0_{n \times \w}&  X_0 &\ldots &X_{L-1}\end{bmatrix} S_L(\xi) $, and $R(\xi)=\begin{bmatrix}  R_0 &\ldots &R_{L-1}&R_{L}\end{bmatrix}S_L(\xi)$. The final part of the claim follows in straightforward way. \end{IEEEproof}

\begin{IEEEproof}[\bf Proof of Th. \ref{th:LMI4MLF}] 
If solutions $\overline{K}_k$, $\overline{Y}_k$ to (\ref{eq:LMI4MLF1}) exist,  multiplying on the left by $S_L(\zeta)^\top$ defined as in the proof of Prop. \ref{prop:PLEviaLMI} and on the right by $S_L(\eta)$, and defining $\Psi_k(\zeta,\eta):=X_k(\zeta)^\top \overline{K}_kX_k(\eta)$ and $Y_k(\xi):=\overline{Y}_kX_k(\xi)$ we obtain $(\zeta+\eta)\Psi(\zeta,\eta)-Y(\zeta)^\top R(\eta)-R(\zeta)^\top Y(\eta)=\Phi_i(\zeta,\eta)$. Since $Y_i$ is $R$-canonical, it follows from Th. \ref{th:LyapfromQDF} that also $\Phi(\zeta,\eta)$ is,  and consequently $\overline{F}_i$ exist as claimed. Now observe that the first inequality in (\ref{eq:LMI4MLF2}) is equivalent with $V_k^\top \overline{F}_k V_k<0$ and thus it implies $Q_{\Phi_k}(w)=\frac{d}{dt}Q_{\Psi_k}(w)<0$ for all $w\in\B_i$. Applying Th. \ref{th:LyapfromQDF} we conclude that $Q_{\Psi_k}$ is a Lyapunov function for $\B_k$.  The second LMI in (\ref{eq:LMI4MLF2}) implies condition $3.$ of Th. \ref{th:MLFSLDS}. 
\end{IEEEproof}

\begin{IEEEproof}[\bf Proof of Prop. \ref{prop:SPR&PLE}]
From the strict positive-realness of $ND^{-1}$ (see Def. \ref{Def:PRM}) and the fact that $D$ is Hurwitz conclude that $N(-j\omega)^\top D(j\omega)+D(-j\omega)^\top N(j\omega)>0$ for all $\omega\in\R$. The existence of $Q$ then follows from standard arguments in polynomial spectral factorisation.  

That $\Psi$ is a polynomial matrix follows from Th. 3.1 of \cite{QDF}. Since $\mbox{\rm rank}~\mbox{\rm col}(D(\lambda), Q(\lambda))=\w$ for all $\lambda\in\C$, $\frac{d}{dt} Q_\Psi(w)<0$ for all $w\in\ker D\derb$, $w\neq 0$. Apply Th. \ref{th:LyapfromQDF} to conclude that $Q_\Psi(w)>0$ for all nonzero $w\in\ker D\derb$. This proves that $\Psi$ is a Lyapunov function for $\ker D\derb$. That $\Psi$ is $D$-canonical, and that $QD^{-1}$ is strictly proper, follow from strict properness of $ND^{-1}$ and Th. \ref{th:LyapfromQDF}. 

We  prove the second part of the claim. Use Prop. 4.10 of \cite{QDF} to conclude that since $\Psi$ is $D$-canonical, it is also $\geq 0$. Denote $\Psi^\prime:=\Psi \mod N$. Since $Q_\Psi(w)=Q_{\Psi^{\prime}}(w)$ for all $w\in\ {\ker~N\derb}$, it follows that $Q_{\Psi^{\prime}}\geq 0$ also along $\ker~N\derb$. We now show that $\frac{d}{dt} Q_{\Psi^{\prime}}$ is negative along $\ker~N\derb$. To do so it suffices to show that $\mbox{\rm col}(Q(\lambda), N(\lambda))=\w$ for all $\lambda\in\C$. Assume by 
contradiction that there exists $\lambda\in\C$ and a corresponding $v\in\C^\w$, $v\neq 0$, such that $Q(\lambda)v=0$ and $N(\lambda)v=0$. Substitute $\zeta=-\lambda$, $\eta=\lambda$ in the PLE, obtaining $D(-\lambda)^\top N(\lambda)+N(-\lambda)^\top D(\lambda)=Q(-\lambda)^\top Q(\lambda)$. Multiply on the right by $v$; it follows that $N(-\lambda)^\top D(\lambda)v=0$. Since $N$ is Hurwitz, this implies $D(\lambda)v=0$, but this contradicts the assumption $\mbox{\rm rank}~\mbox{\rm col}(D(\lambda), Q(\lambda))=\w$.  
\end{IEEEproof}


\begin{IEEEproof}[\bf Proof of Lemma \ref{lemma:state}] That  $n_2<n_1$ follows from $R_2R_1^{-1}$ being strictly proper. 

To prove the claim on $X_1$ defined by (\ref{stateB1}), define 
 $\mathfrak{X}_i:=\{ f\in\mathbb{R}^{1\times \tt w}[\xi] ~|~ fR_i^{-1} \mbox{ is strictly  proper}\}$, $i=1,2$; we now show that $\mathfrak{X}_2 \subset \mathfrak{X}_1$. Observe that $fR_2^{-1}\cdot R_2 R_1^{-1}=f R_1^{-1}$; since both $f R_2^{-1}$ and $R_2 R_1^{-1}$ are strictly proper, so is their product. Consequently, $f\in\mathfrak{X}_1$. Observe that $\mathfrak{X}_i$ is the state space of $\B_i$, $i=1,2$ (see App. \ref{app:backmater:statemaps}). 

Arrange the vectors of a  basis for $\mathfrak{X}_2$ in $X_2\in\R^{n_2\times \w}[\xi]$; then $X_2\left(\frac{d}{dt}\right)$ is a  state map for $\mathfrak{B}_2$. Complete $X_2$ with $X_1^\prime\in\R^{(n_1-n_2)\times \w}[\xi]$ to form a basis of $\mathfrak{X}_1$; this defines a state map for $\B_1$. 

Since each row of $X_1^\prime~ \mbox{\rm mod}~ R_2$ belongs to $\mathfrak{X}_2$, it can be written as a linear combination of the rows of $X_2$. This proves that $\Pi$ exists. \end{IEEEproof}

\begin{IEEEproof}[\bf Proof of Theorem \ref{th_01}]
The existence of $Q\in\R^{\bullet\times\w}[\xi]$ and the $R_1$-canonicity of $\Psi_1$ follow from Prop. \ref{prop:SPR&PLE}. To prove that $\Psi_1$ and $\Psi_2:=\Psi_1~\mbox{\rm mod}~ R_2$ yield a MLF we show that:
\begin{itemize}
\item[\mbox{\bf C1.}] $Q_{\Psi_1} \stackrel{\mathfrak{B}_1}\ge0$ and $\frac{d}{dt}Q_{\Psi_1} \stackrel{\mathfrak{B}_1}< 0$;
\item[\mbox{\bf C2.}] $Q_{\Psi_2} \stackrel{\mathfrak{B}_2}\ge0$ and $\frac{d}{dt}Q_{\Psi_2} \stackrel{\mathfrak{B}_2}< 0$;
\item[\mbox{\bf C3.}] The multiple functional associated with $\Psi_1$ and $\Psi_2$ does not increase at switching instants.
\end{itemize}

Conditions {\bf C1} and {\bf C2} follow from Prop. \ref{prop:SPR&PLE}. 

To prove {\bf C3}, we first define the \emph{coefficient matrices} of $\Psi_1$ and $\Psi_2$. Since $\Psi_1$ is $R_1$-canonical, it can be written as $X_1(\zeta)^\top \widetilde{\Psi}_1 X_1(\eta)$ for some coefficient matrix $\widetilde{\Psi}_1\in\R^{n_1\times n_1}$.  Since $Q R_1^{-1}$ is strictly proper, it follows (see Th. \ref{th:LyapfromQDF}) that $Q_{\Psi_1}\overset{\B}{>}0$ and since $X_1$ is a minimal state map for $\B_1$ it follows that $\widetilde{\Psi}_1>0$. Note that $\col(X_2(\xi),X^\prime_1(\xi))~\mbox{\rm mod}~R_2=\col(X_2(\xi)~\mbox{\rm mod}~R_2,X^\prime_1(\xi)~\mbox{\rm mod}~R_2)=\col(X_2(\xi),\Pi X_2(\xi))$. Consequently (see Prop. 4.9 of \cite{QDF}), 
\[
\Psi_1(\zeta,\eta)\,\mbox{mod}\,R_2=\begin{bmatrix}X_2(\zeta)^\top & X_2(\zeta)^\top \Pi^\top \end{bmatrix}\widetilde{\Psi}_1 \begin{bmatrix} X_2(\eta) \\ \Pi X_2(\eta) \end{bmatrix}\;,
\]
from which it follows that the coefficient matrix of $\Psi_2$ is $\widetilde{\Psi}_2=\col(I_{n_2},\Pi)^\top~ \widetilde{\Psi}_1~\col(I_{n_2},\Pi)$. 

We  prove {\bf C3} showing that $\widetilde{\Psi}_1$ and $\widetilde{\Psi}_2$ satisfy some structural properties. We begin proving the following linear algebra result. 
\begin{lemma}\label{lemma:fundstructure}
Let $\Pi\in\R^{(n_1-n_2)\times n_2}$, and $\widetilde{\Psi}_1=\widetilde{\Psi}_1^\top\in\R^{n_1\times n_1}$. Assume  $\widetilde{\Psi}_1>0$, and define $\widetilde{\Psi}_2^e:=\begin{bmatrix} I_{n_2}& \Pi^\top\\ (0_{n_1-n_2\times n_2})&(0_{n_1-n_2\times n_1-n_2}) \end{bmatrix} \widetilde{\Psi}_1 \begin{bmatrix} I_{n_2}&(0_{n_2\times n_1-n_2})\\ \Pi &(0_{n_1-n_2\times n_1-n_2})\end{bmatrix}$.  $\widetilde{\Psi}_1\geq \widetilde{\Psi}_2^e$ if and only if  there exist $\Psi_{11}\in\R^{n_2 \times n_2}$, $\Psi_{12}\in\R^{n_2 \times (n_1-n_2)}$ and $\Psi_{22}\in\R^{(n_1-n_2) \times (n_1-n_2)}$ such that $\widetilde{\Psi}_1=
\begin{bmatrix}
\Psi_{11} & -\Pi^\top \Psi_{22}\\
 -\Psi_{22} \Pi & \Psi_{22}
\end{bmatrix}$. 
\end{lemma}
\begin{IEEEproof}[\bf Proof of Lemma \ref{lemma:fundstructure}]
Partition $\widetilde{\Psi}_1=:\begin{bmatrix} \Psi_{11} & \Psi_{12} \\ \Psi_{12}^\top & \Psi_{22} \end{bmatrix}$, with $\Psi_{11}\in\R^{n_2 \times n_2}$, $\Psi_{12}\in\R^{n_2 \times (n_1-n_2)}$ and $\Psi_{22}\in\R^{(n_1-n_2) \times (n_1-n_2)}$. Straightforward manipulations show that $\widetilde{\Psi}_1\geq \widetilde{\Psi}_2^e$ iff
\[
\begin{bmatrix} -(\Psi_{12}+\Pi^\top \Psi_{22})\Psi_{22}^{-1}(\Psi_{12}^\top + \Psi_{22}\Pi) & 0 \\ 0 & \Psi_{22}  \end{bmatrix}\ge 0 \;.
\]
Now $\Psi_{22}>0$, since $\widetilde{\Psi}_1>0$;  thus the inequality holds iff $\Psi_{12}^\top=-\Psi_{22}\Pi$. 
\end{IEEEproof}

We aim to show that Lemma \ref{lemma:fundstructure} holds for the  coefficient matrix of $\Psi_1$ and the  $\Pi$ arising from the standard gluing conditions. To this purpose we  first prove the following result.
\begin{lemma}\label{lemma_R2}
Define $K:=\lim_{\xi\rightarrow \infty} \xi X_1^\prime(\xi) R_1(\xi)^{-1}$; then $K\in \R^{(n_1-n_2)\times \tt w}$. Moreover, partition $\widetilde{\Psi}_1$ as $\widetilde{\Psi}_1=:\begin{bmatrix} \Psi_{11} & \Psi_{12} \\ \Psi_{12}^\top & \Psi_{22} \end{bmatrix}$, with $\Psi_{11}\in\R^{n_2 \times n_2}$, $\Psi_{12}\in\R^{n_2 \times (n_1-n_2)}$ and $\Psi_{22}\in\R^{(n_1-n_2) \times (n_1-n_2)}$. Then $R_2(\xi)=K^\top\left(\Psi_{12}^\top X_2(\xi) + \Psi_{22}X_1^\prime(\xi) \right)$. 
\end{lemma}
\noindent\begin{IEEEproof}[\bf Proof of Lemma \ref{lemma_R2}] That the limit is finite follows from $X_1^\prime R_1^{-1}$ being strictly proper. To prove the rest, recall from App. \ref{app:backmater:statemaps} that there exist $A_1\in\R^{n_1\times n_1}$, $F_1\in\R^{n_1\times \w}$ such that 
\begin{equation}
\xi X_1(\xi)=A_1 X_1(\xi)+F_1(\xi) R_1(\xi)\;.
\label{poly-ss}
\end{equation}
Multiply both sides of (\ref{poly-ss}) by $R_1^{-1}$, and take the limit for $\xi\rightarrow \infty$.  Since $R_2R_1^{-1}$ is strictly proper and $X_2(\xi)$ is a state map for $\B_2$, it follows that $\lim_{\xi \rightarrow \infty} \xi X_2(\xi)R_1(\xi)^{-1}=0_{n_2\times \w}$. Moreover, $\lim_{\xi \rightarrow \infty}  X_1(\xi)R_1(\xi)^{-1}=0_{n_1\times \w}$. Consequently $F_1$ is constant, and 
\[
F_1= \lim_{\xi \rightarrow \infty} \col(0_{n_2\times \w} ,\xi X_1^\prime (\xi)R_1(\xi)^{-1})=\col(0_{n_2\times \w}, K)\; .
\] 
The claim on $R_2$ now follows from Prop. 4.3 of \cite{PLE}. 
\end{IEEEproof}

From Lemma \ref{lemma_R2} and the fact that $R_2$ is square and nonsingular, it follows that $K^\top$ is of full row rank, and consequently $n_1-n_2\geq \w$. We now prove that $K$  is square, thus nonsingular. 
\begin{lemma}\label{lemma:n1-n2}
$\deg(\det(R_1))-\deg(\det(R_2))=n_1-n_2=\tt w$, and consequently $K$ is nonsingular.
\end{lemma}
\begin{IEEEproof}[\bf Proof of Lemma \ref{lemma:n1-n2}]
We prove the first part of the claim, well-known in the scalar case, but for whose multivariable version we have failed to find a proof in the literature. 

Let $U\in\R^{\w\times\w}[\xi]$ be a unimodular matrix such that $R_1^\prime:=R_1U$ is column reduced (see sect. 6.3.2 of \cite{kailath}); define $R_2^\prime:=R_2U$. Observe that $R_2^\prime R_1^{\prime -1}=R_2R_1^{-1}$; moreover $n_1=\deg(\det(R_1^\prime))=\deg(\det(R_1))$ and $n_2=\deg(\det(R_2))=\deg(\det(R_2^\prime))$. Thus w.l.o.g. we  prove the claim for $R_2^\prime R_1^{\prime -1}$. 

Define $\mathfrak{X}_{1}^\prime:=\{ f\in\R^{1\times\w}[\xi] \mid f R_1^{\prime -1} \mbox{is strictly proper}\}$ and  similarly $\mathfrak{X}_{2}^\prime$; it is straightforward to see that 
$\mathfrak{X}_{i}^\prime$ equals $\mathfrak{X}_{i}$ defined as in Lemma \ref{lemma:state}, $i=1,2$. Denote the degree of the $i$-th column of  $R_1^\prime$ by $\delta_i^1$ and that of the $i$-th column of  $R_2^\prime$ by $\delta_i^2$, $i=1,\ldots,\w$; strict properness yields $\delta_i^1>\delta_i^2$, $i=1,\ldots,\w$. A basis for $\mathfrak{X}_{1}^\prime$ is $e_i \xi^{k}$, $k=1,\ldots,\delta_k^1-1$, $i=1,\ldots, \w$, where $e_i$ is the $i$-th vector of the canonical basis for $\R^{1\times \w}$. A straightforward argument proves that  these vectors can be arranged in a matrix $X(\xi)=\col(X_2(\xi),X_1^\prime(\xi))$ so that the $n_2$ rows of $X_2$ span  $\mathfrak{X}_{2}^\prime$ and those of $X_1^\prime$ span its complement in $\mathfrak{X}_{1}^\prime$.  Permute the rows of $X_1^\prime$ so that $e_i\xi^{\delta_i^1-1}$, $i=1,\ldots,\tt w$, are its last $\tt w$ rows. 

An analogous of (\ref{poly-ss}) holds for $R_1^\prime$; given the arrangement of the basis vectors for $\mathfrak{X}_1^\prime$, it is straightforward to verify that the last $\tt w$ rows of $K$ contain the inverse of the highest column coefficient matrix of $R_1$, while its first $n_1-n_2-\tt w$ rows are equal to zero, i.e. $K^\top=\begin{bmatrix} 0_{(n_1-n_2-\w)\times \w} &{K^\prime}^\top \end{bmatrix}$, with $K^{\prime}\in\R^{\tt w \times w}$ nonsingular. 

Now let $\Psi_1^\prime$ be a storage function for $R_2^\prime R_1^{\prime -1}$ with the same properties as $\Psi_1$ in the statement of Th.  \ref{th_01}; we denote with $\Psi_{ij}^\prime$, $i,j=1,2$ the block submatrices arising from a partition of its coefficient matrix $\widetilde{\Psi^\prime}_1$ as in Lemma \ref{lemma_R2}. Use the formula for $R_2^\prime(\xi)$ established in Lemma \ref{lemma_R2} to conclude that  $R_2^\prime(\xi)=K'^\top \Psi_{12}^{\prime\top} X_2(\xi) + K^{\prime\top}  \begin{bmatrix} \Psi_{22}^{\prime\prime} & \Psi_{22}^{\prime \prime\prime}\end{bmatrix}X_1^\prime (\xi)$, 
where $\Psi_{12}^{\prime\top}\in\R^{{\tt w} \times n_2}$, $\begin{bmatrix} \Psi_{22}^{\prime\prime} & \Psi_{22}^{\prime \prime\prime}\end{bmatrix}\in\R^{\w\times (n_1-n_2)}$, and $\Psi_{22}^{\prime \prime\prime}$ has $\w$ columns. $\widetilde{\Psi^\prime}_1>0$ implies $\Psi_{22}^{\prime\prime\prime}>0$; thus the highest column coefficient matrix of $R_2(\xi)$ is $K^\prime \Psi_{22}^{\prime\prime\prime}$ and it is nonsingular. Thus also $R_2^\prime(\xi)$ is column reduced; moreover, its column degrees are $\delta_i^1 -1$, $i=1,\ldots,\w$. From this it follows that $\deg\det(R_2^\prime)=\sum_{i=1}^\w (\delta_i^1-1)=(\sum_{i=1}^\w \delta_i^1)-\w=n_1-\w$. The claim is proved.
\end{IEEEproof}

We resume the proof of Th. \ref{th_01}. From the formula for $R_2(\xi)$ proved in  Lemma \ref{lemma_R2} it follows that   
\begin{equation}\label{modR2}
0=R_2(\xi)\,\mbox{mod}\, R_2 = K^\top\left(\Psi_{12}^\top X_2(\xi) + \Psi_{22} X_1^\prime(\xi) \right)\mbox{mod}\, R_2=K^\top\left(\Psi_{12}^\top  + \Psi_{22}\Pi \right) X_2(\xi)\;.
\end{equation}
The rows of $X_2(\xi)$ are linearly independent over $\mathbb{R}$, since $X_2$ is a minimal state map. Consequently (\ref{modR2}) implies $K^\top(\Psi_{12}^\top + \Psi_{22}\Pi)=0$, and since $K$ is nonsingular by Lemma \ref{lemma:n1-n2}, we conclude that $\Psi_{12}^\top + \Psi_{22}\Pi=0$. Thus the coefficient matrix of $\Psi_1$ is structured as in Lemma \ref{lemma:fundstructure}. 

We now show that this structure implies that condition {\bf C3} holds. 
Consider first a switch from $\mathfrak{B}_1$ to $\mathfrak{B}_2$ at $t_k$. Taking the standard gluing conditions into account, $Q_{\Psi_1}(w)(t_k^-)\geq Q_{\Psi_{2}}(w)(t_k^+)$ if and only if
\begin{eqnarray}
&& \begin{bmatrix} X_2(\frac{d}{dt})w(t_k^-) \\ X_1^\prime (\frac{d}{dt})w(t_k^-) \end{bmatrix}^\top \widetilde{\Psi}_1\begin{bmatrix} X_2(\frac{d}{dt})w(t_k^-) \\ X_1^\prime (\frac{d}{dt})w(t_k^-) \end{bmatrix} - \begin{bmatrix} X_2(\frac{d}{dt})w(t_k^+) \\  \Pi X_2(\frac{d}{dt})w(t_k^+) \end{bmatrix}^\top \widetilde{\Psi}_1 \begin{bmatrix} X_2(\frac{d}{dt})w(t_k^+) \\  \Pi X_2(\frac{d}{dt})w(t_k^+) \end{bmatrix} \nonumber \\
&&=\begin{bmatrix} X_2(\frac{d}{dt})w(t_k^-) \\ X_1^\prime (\frac{d}{dt})w(t_k^-) \end{bmatrix}^\top \left(  \widetilde{\Psi}_1 - \begin{bmatrix} I_{n_2} & \Pi^\top \\ 0& 0  \end{bmatrix} \widetilde{\Psi}_1 \begin{bmatrix} I_{n_2} & 0 \\ \Pi & 0  \end{bmatrix} \right) \begin{bmatrix} X_2(\frac{d}{dt})w(t_k^-) \\ X_1^\prime (\frac{d}{dt})w(t_k^-) \end{bmatrix} \ge 0\; .
\label{ineq1}
\end{eqnarray}
Since the matrix between brackets is semidefinite positive (see Lemma \ref{lemma:fundstructure}), (\ref{ineq1}) is  satisfied.

It is straightforward to check that in a switch from $\mathfrak{B}_2$ to $\mathfrak{B}_1$ the value of the multi-functional is the same before and after the switch. The theorem is proved.
\end{IEEEproof}


\begin{IEEEproof}[\bf Proof of Th. \ref{th_comp}] W.l.o.g.  assume that $Q_\Psi$ is $R_1$-canonical; then by Lemma \ref{lemma:state}, given a minimal state map $X_1\derb$ for $\mathfrak{B}_1$ as in (\ref{stateB1}) there exists $\widetilde{\Psi}=\widetilde{\Psi}^\top\in\R^{n_1\times n_1}$ such that $\Psi(\zeta,\eta)=X_1(\zeta)^\top \widetilde{\Psi} X_1(\eta)$. Partition $\widetilde{\Psi}$ as $\widetilde{\Psi}=:\begin{bmatrix} \Psi_{11} & \Psi_{12} \\ \Psi_{12}^\top & \Psi_{22} \end{bmatrix}$ where $\Psi_{11}\in\R^{n_2 \times n_2}$, $\Psi_{12}\in\R^{n_2 \times (n_1-n_2)}$ and $\Psi_{22}\in\R^{(n_1-n_2) \times (n_1-n_2)}$. At a switch from $\mathfrak{B}_1$ to $\mathfrak{B}_2$ at $t_k$ the inequality (\ref{ineq1}) holds in particular for a switching signal $s(t)=1$ for $t\leq t_k$, $s(t)=2$ for $t>t_k$. Since for every choice of $v\in\R^{n_1}$ there exists a trajectory $w\in\B_1|_{(-\infty,0]}$ s.t. $\left( X_1\derb w \right)(0^-)=v$, using Lemma 
\ref{lemma:fundstructure} we conclude that (\ref{ineq1}) holds, then $\Psi_{12}^\top+\Psi_{22}\Pi=0$. Consequently, 
\begin{equation}\label{eq:jump}
\widetilde{\Psi}=\begin{bmatrix} \Psi_{11} & -\Pi \Psi_{22} \\ -\Psi_{22}\Pi & \Psi_{22} \end{bmatrix}= \begin{bmatrix} \widetilde{\Psi}^\prime & 0 \\ 0 & 0 \end{bmatrix} +
 \begin{bmatrix} \Pi^\top  \\ -I_{n_1-n_2} \end{bmatrix}  \Psi_{22}  \begin{bmatrix} \Pi & -I_{n_1-n_2} \end{bmatrix} \;,
\end{equation}
where $\widetilde{\Psi}^\prime:=\Psi_{11}-\Pi^\top \Psi_{22} \Pi$. Pre- and post-multiply (\ref{eq:jump}) by $X_1(\zeta)^\top$ and $X_1(\eta)$ to obtain
\begin{equation}\label{LyapB2}
\Psi(\zeta,\eta)=\underset{=:\Psi'(\zeta,\eta)}{\underbrace{X_2(\zeta)^\top \tilde{\Psi }' X_2(\eta)}} +  X_1(\zeta)^\top \begin{bmatrix} \Pi^\top  \\ -I_{(n_1-n_2)} \end{bmatrix}  \Psi_{22}  \begin{bmatrix} \Pi & -I_{(n_1-n_2)} \end{bmatrix} X_1(\eta)\;.
\end{equation}
Since $\Psi_1$ is a Lyapunov function for $\ker~R_1\derb$, there exists $V\in\R^{\w\times\w}[\xi]$ such that $(\zeta+\eta)\Psi_1(\zeta,\eta)=-Q(\zeta)^\top Q(\eta)+V(\zeta)^\top R_1(\eta)+R_1(\zeta)^\top V(\eta)$. We now show that there exists $M\in\R^{\w\times\w}[\xi]$ such that  $V=MR_2$. 

From Prop. 4.3 of \cite{PLE} it follows that $V(\xi)=\lim_{\mu\rightarrow \infty} \mu R_1(\mu)^{-\top} \Psi_1(\mu,\xi)$; substituting  (\ref{LyapB2}) in this expression we obtain
\begin{eqnarray*}\label{eq:faster}
V(\xi)&=&\lim_{\mu\rightarrow \infty} \Big( \mu R_1(\mu)^{-\top} X_2(\mu)^\top \tilde{\Psi }' X_2(\eta)\nonumber\\&& +\mu R_1(\mu)^{-\top}   X_1(\mu)^\top \begin{bmatrix} \Pi^\top  \\ -I_{(n_1-n_2)} \end{bmatrix}  \Psi_{22}  \begin{bmatrix} \Pi & -I_{(n_1-n_2)} \end{bmatrix} X_1(\eta) \Big) \; .
\end{eqnarray*}
Since $R_2R_1^{-1}$ is strictly proper, the first term goes to zero. Now $\begin{bmatrix}\Pi&-I_{n_1-n_2} \end{bmatrix} X_1(\xi)=-X_1^\prime(\xi)+\Pi X_2(\xi)$ and consequently 
\begin{eqnarray*}
V(\xi)=&&-\mu R_1(\mu)^{-\top} X_1^{\prime\top}(\mu)\Psi_{22}\begin{bmatrix} \Pi & -I_{(n_1-n_2)} \end{bmatrix} X_1(\xi)\\&& + \lim_{\mu\rightarrow \infty} \underset{\rightarrow 0}{\underbrace{\mu R_1(\mu)^{-\top}   X_2(\mu)^\top}} \Pi^\top \Psi_{22}  \begin{bmatrix} \Pi & -I_{(n_1-n_2)} \end{bmatrix} X_1(\xi)\\&&=-\begin{bmatrix}0_{(n_1-n_2)\times \w}& K^{\prime\top}\end{bmatrix}\Psi_{22} \begin{bmatrix} \Pi & -I_{(n_1-n_2)} \end{bmatrix} X_1(\xi) \; ,
\end{eqnarray*} 
where $K^\prime\in\R^{\w\times\w}$ is a nonsingular matrix, as proved in Lemma \ref{lemma_R2} and \ref{lemma:n1-n2}. That $V$ has the right factor $R_2$ follows from the following argument. 
Observe that $\begin{bmatrix} \Pi & -I_{(n_1-n_2)} \end{bmatrix}  \begin{bmatrix} X_2(\xi) \\X_1^\prime (\xi) \end{bmatrix}= X_1^\prime(\xi)\,\mbox{mod}\,R_2- X_1^\prime(\xi)$. Write $X_1^\prime(\xi)R_2(\xi)^{-1}=P(\xi)+S(\xi)$, with $S(\xi)$ a strictly proper polynomial matrix and $P\in\R^{(n_1-n_2)\times \tt w}[\xi]$; then $\Pi X_2(\xi) -X_1^\prime (\xi)= X_1^\prime(\xi)-P(\xi)R_2(\xi)- X_1^\prime(\xi) 
=-P(\xi)R_2(\xi)$. 
This proves that $V(\xi)=\begin{bmatrix}0_{(n_1-n_2)\times \w}& K^{\prime\top}\end{bmatrix}\Psi_{22}P(\xi)R_2(\xi)=:M(\xi)R_2(\xi)$.

The equality  $(\zeta+\eta)\Psi_1(\zeta,\eta)=-Q(\zeta)^\top Q(\eta)+R_2(\zeta)^\top M(\zeta)^\top R_1(\eta)+R_1(\zeta)^\top M(\eta) R_2(\eta)$, together with $\rank~Q(j\omega)=\w$ for all $\omega\in\R$ and  $R_1$ being  Hurwitz, prove strict positive-realness of $MR_2R_1^{-1}$. That $MR_2R_1^{-1}$ is  strictly proper follows from $QR_1^{-1}$ being strictly proper and Th. \ref{th:LyapfromQDF}. This concludes the proof. \bigskip
\end{IEEEproof}

\bibliographystyle{plain}

\bibliography{biblio}

\end{document}